\documentclass[a4paper,11pt,reqno]{amsart}
\usepackage{amsmath}
\usepackage{amstext}
\usepackage{amsbsy}
\usepackage{amsopn}
\usepackage{upref}
\usepackage{amsthm}
\usepackage{amsfonts}
\usepackage{amssymb}
\usepackage{mathrsfs}
\usepackage{times}
\allowdisplaybreaks
 \newtheorem{theorem}{Theorem}[section]

 \newtheorem{lemma}[theorem]{Lemma}
\theoremstyle{definition}
 
\theoremstyle{remark}
 
\newcommand{\p}{\partial}

\newcommand{\re}{\operatorname{Re}}
\newcommand{\im}{\operatorname{Im}}
\begin{document}
\title[Resolvent estimates and dispersive equations]
      {Resolvent estimates related with 
       \\
       a class of dispersive equations}
\author[H.~Chihara]{Hiroyuki CHIHARA}
\address{Mathematical Institute,  
         Tohoku University, 
         Sendai 980-8578, Japan}
\email{chihara@math.tohoku.ac.jp}
\thanks{The author was supported by 
the JSPS Grant-in-Aid for Scientific Research \#17540140.}
\begin{abstract}
We present a simple proof of 
the resolvent estimates of 
elliptic Fourier multipliers 
on the Euclidean space, 
and apply them to the analysis of 
time-global and spatially-local smoothing estimates 
of a class of dispersive equations. 
For this purpose we study in detail the properties of 
the restriction of Fourier transform on 
the unit cotangent sphere associated 
with the symbols of multipliers. 
\end{abstract}
\maketitle
\section{Introduction}
\label{section:introduction}
This paper is concerned with resolvent estimates of elliptic operators 
on the Euclidean space with constant coefficients. These estimates are
equivalent to smoothing properties of solutions to corresponding
dispersive evolution equations. 
\par
For 
$x=(x_1,\dotsc,x_n)\in\mathbb{R}^n$ 
and 
$\xi=(\xi_1,\dotsc,\xi_n)\in\mathbb{R}^n$, 
set $x\cdot\xi=x_1\xi_1+\dotsb+x_n\xi_n$ and 
$\lvert\xi\rvert=\sqrt{\xi\cdot\xi}$. 
Let 
$a(\xi){\in}C(\mathbb{R}^n){\cap}C^\infty(\mathbb{R}^n\setminus\{0\})$
be a positively homogeneous function of degree one. 
Suppose that $a(\xi)>0$ for $\xi\ne0$. 
It follows that 
$$
a^\prime(\xi)
=
\left(
\frac{\p{a}}{\p\xi_1}(\xi),\dotsc,\frac{\p{a}}{\p\xi_n}(\xi)
\right)
\ne0
\quad\text{for}\quad
\xi\ne0
$$
since $a(\xi)=a^\prime(\xi)\cdot\xi$. 
Set $p(\xi)=a(\xi)^m$ for some fixed number $m>1$. 
\par
Consider the initial value problem of the form 
\begin{alignat}{2}
  D_tu-p(D_x)u
& =f(t,x) 
& 
  \quad\text{in}\quad
& \mathbb{R}^{1+n},
\label{equation:pde} 
\\
  u(0,x)
& =
  \phi(x)
& 
  \quad\text{in}\quad
& \mathbb{R}^{n},
\label{equation:data} 
\end{alignat}
where $u(t,x)$ is an unknown function of 
$(t,x)=(t,x_1,\dotsc,x_n)\in\mathbb{R}^{1+n}$, 
$f(t,x)$ and $\phi(x)$ are given functions, 
$i=\sqrt{-1}$, 
$D_t=-i\p/\p{t}$, 
$D_x=-i\p/\p{x}$, 
$\p/\p{x}=(\p/\p{x_1},\dotsc,\p/\p{x_n})$, 
and the operator $p(D_x)$ is defined by 
$$
p(D_x)v(x)
=
\frac{1}{(2\pi)^n}
\iint_{\mathbb{R}^n\times\mathbb{R}^n}
e^{i(x-y)\cdot\xi}p(\xi)v(y)
dyd\xi
$$
for an appropriate function $v(x)$. 
Since $p(\xi)$ is real-valued, 
the initial value problem 
\eqref{equation:pde}-\eqref{equation:data} is $L^2$-well-posed, 
that is, 
for any $\phi{\in}L^2(\mathbb{R}^n)$ 
and for any $f{\in}L^1_{\text{loc}}(\mathbb{R};L^2(\mathbb{R}^n))$, 
\eqref{equation:pde}-\eqref{equation:data} possesses a unique solution 
$u{\in}C(\mathbb{R};L^2(\mathbb{R}^n))$. 
Here $L^q$ and $L^q_{\text{loc}}$ denote a usual Lebesgue space and 
its local space respectively for $q\in[1,\infty]$, 
and $C(\mathbb{R};L^2(\mathbb{R}^n))$ is the set of all 
$L^2(\mathbb{R}^n)$-valued continuous functions on $\mathbb{R}$.  
Moreover, the unique solution $u$ is explicitly given by 
\begin{align*}
  u(t,x)
& =
  e^{itp(D_x)}\phi(x)+iGf(t,x),
\\
  e^{itp(D_x)}\phi(x)
& =
  \frac{1}{(2\pi)^n}
  \iint_{\mathbb{R}^n\times\mathbb{R}^n}
  e^{i(x-y)\cdot\xi}e^{itp(\xi)}v(y)
  dyd\xi,
\\
  Gf(t,x)
& =
  \int_0^t
  e^{i(t-s)p(D_x)}f(s,x)
  ds.
\end{align*}
\par
A typical example of \eqref{equation:pde} is the Schr\"odinger evolution 
equation of a free particle, which is the case $p(\xi)=\lvert\xi\rvert^2$. 
It is well-known that the solution to the free Schr\"odinger evolution
equation on $\mathbb{R}^n$ gains extra smoothness in comparison with the
initial data and the forcing term. 
This mathematical phenomenon is called local smoothing effect or local
smoothing property. 
In the last two decades, smoothing properties of solutions to more
general dispersive partial differential equations and 
their applications have been vigorously investigated. 
See, e.g.,  
\cite{chihara}, 
\cite{doi}, 
\cite{hoshiro1}, 
\cite{hoshiro2}, 
\cite{hoshiro3}, 
\cite{ky}, 
\cite{kpv}, 
\cite{lp}, 
\cite{rs}, 
\cite{sugimoto}, 
\cite{walther}, 
\cite{watanabe} 
and references therein. 
\par
In \cite{doi} Doi deeply studied the relationship between the behavior
of the geodesic flow and the smoothing effect of the Schr\"odinger
evolution equation on complete Riemannian manifolds. 
Roughly speaking, he proved that the smoothing effect occurs if and only
if all the geodesics go to ``infinity''. 
In other words, if there exists a trapped geodesic, then the smoothing
effect breaks down. 
For more general dispersive equations, the smoothing effect depends on
the behavior of the Hamilton flow generated by the principal symbol of 
the equations. 
Consider dispersive equations with constant coefficients of the form
\begin{equation}
D_tu-q(D_x)u=f(t,x)
\quad\text{in}\quad
\mathbb{R}^{1+n}, 
\label{equation:hoshiro}
\end{equation}
where $q(\xi)$ is a real polynomial of order $m>1$. 
Let $q_m(\xi)$ be the principal symbol of $q(D_x)$. 
$q_m(\xi)$ generates the Hamilton flow 
$\{(x+tq_m^\prime(\xi),\xi)\}_{t\in\mathbb{R}}$ 
for 
$(x,\xi)\in\mathbb{R}^n\times(\mathbb{R}^n\setminus\{0\})$. 
Hoshiro proved that the smoothing effect of solutions
to the IVP for \eqref{equation:hoshiro} occurs if and only if for $\xi\ne0$
$$
\lvert{x+tq_m^\prime(\xi)}\rvert 
\longrightarrow 
\infty 
\quad\text{as}\quad
t \rightarrow \pm\infty.
$$
This condition is equivalent to $q_m^\prime(\xi)\ne0$ for $\xi\ne0$. 
See \cite{hoshiro3} for the detail. 
\par
There are some expressions of smoothing estimates. 
Let 
$\Delta=-\lvert{D_x}\rvert^2$ 
and 
$\langle{x}\rangle=\sqrt{1+\lvert{x}\rvert^2}$. 
Throughout this paper, different positive constants are denoted by the
same letter $C$. 
In \cite{sugimoto} Sugimoto classified these estimates into three types 
as follows. 
\begin{description}
\item[TYPE-I\ ] 
Let $\delta>1/2$. 
Then 
\begin{align}
  \lVert
  \langle{x}\rangle^{-\delta}
  \lvert{D_x}\rvert^{1/2}
  e^{-it\Delta}\phi
  \rVert_{L^2(\mathbb{R}^{1+n})}
& \leqslant
  C
  \lVert\phi\rVert_{L^2(\mathbb{R}^n)},
\label{equation:sugimoto1}
\\
  \left\lVert
  \langle{x}\rangle^{-\delta}
  \lvert{D_x}\rvert
  \int_0^t
  e^{-i(t-s)\Delta}f(s)
  ds
  \right\rVert_{L^2(\mathbb{R}^{1+n})}
& \leqslant
  C
  \lVert
  \langle{x}\rangle^{\delta}f
  \rVert_{L^2(\mathbb{R}^{1+n})}.
\label{equation:sugimoto2}
\end{align}
\item[TYPE-II\ ] 
Suppose $n\geqslant3$. Then 
\begin{align}
  \lVert
  \langle{x}\rangle^{-1}
  \langle{D_x}\rangle^{1/2}
  e^{-it\Delta}\phi
  \rVert_{L^2(\mathbb{R}^{1+n})}
& \leqslant
  C
  \lVert\phi\rVert_{L^2(\mathbb{R}^n)},
\label{equation:sugimoto3}
\\
  \left\lVert
  \langle{x}\rangle^{-1}
  \langle{D_x}\rangle
  \int_0^t
  e^{-i(t-s)\Delta}f(s)
  ds
  \right\rVert_{L^2(\mathbb{R}^{1+n})}
& \leqslant
  C
  \lVert
  \langle{x}\rangle{f}
  \rVert_{L^2(\mathbb{R}^{1+n})}.
\label{equation:sugimoto4}
\end{align}
\end{description}
For TYPE-III, see \cite{ky}, \cite{sugimoto} and \cite{watanabe} for the detail. 
TYPE-I estimates 
\eqref{equation:sugimoto1} 
and 
\eqref{equation:sugimoto2} 
were studied by many authors. 
These inequalities describe time-global and spatially local smoothing effect. 
TYPE-II estimates 
\eqref{equation:sugimoto3} 
and 
\eqref{equation:sugimoto4},  
and TYPE-III estimates were 
first studied by Kato and Yajima in \cite{ky}. 
These inequalities seem to show 
not only smoothing effect 
but also integrability of solutions. 
In \cite{chihara} the author gave the complete generalization of 
\eqref{equation:sugimoto1} and \eqref{equation:sugimoto2}. 
More precisely, he proved that if $p(\xi)$ is a 
real-principal-type homogeneous symbol of degree $m>1$, that is, 
$p(\xi){\in}C^1(\mathbb{R}^n){\cap}C^\infty(\mathbb{R}^n\setminus\{0\})$ 
is real-valued and satisfies 
$$
p^\prime(\xi)\ne0, 
\quad
p(\xi)
=
\lvert\xi\rvert^m
p\left(\frac{\xi}{\lvert\xi\rvert}\right) 
\quad\text{for}\quad
\xi\ne0, 
$$
then the corresponding TYPE-I estimates hold. 
Unfortunately, the proof of this generalization in \cite{chihara} 
is complicated and not comprehensive. 
On the other hand, TYPE-II estimates seem to show the integrability of
solutions related with low-frequency region also. 
From a point of view of Fourier analysis, 
it is very natural to ask 
whether the curvature effect of the level set of
$a(\xi)=\lvert\xi\rvert$ is essential or not. 
\par
There are two purposes of this paper. 
One is to give a simpler proof of the generalization of TYPE-I estimates 
in \cite{chihara} for general elliptic symbols. 
It seems to be very hard to present essentially simpler proof for
more general real-principal-type symbols including nonelliptic ones. 
Another is to consider the influence of the curvature effect of 
the level set of the elliptic symbol on TYPE-II estimates. 
More precisely, 
we will give the complete generalization of TYPE-II estimates 
\eqref{equation:sugimoto3} 
and 
\eqref{equation:sugimoto4}, 
and show that they depend only on 
$p^\prime(\xi)$ and the weight $\langle{x}\rangle^{-m/2}$, 
and are independent of the curvature of the level set of $p(\xi)$. 
Our results are the following.   
\begin{theorem}
\label{theorem:main}
Let $n\geqslant2$. 
\begin{description}
\item[TYPE-I]
Suppose $m>1$ and $\delta>1/2$. Then 
\begin{align}
  \lVert
  \langle{x}\rangle^{-\delta}
  \lvert{D_x}\rvert^{(m-1)/2}
  e^{itp(D_x)}\phi
  \rVert_{L^2(\mathbb{R}^{1+n})}
& \leqslant
  C
  \lVert\phi\rVert_{L^2(\mathbb{R}^n)},
\label{equation:chihara1}
\\
  \lVert
  \langle{x}\rangle^{-\delta}
  \lvert{D_x}\rvert^{m-1}
  Gf
  \rVert_{L^2(\mathbb{R}^{1+n})}
& \leqslant
  C
  \lVert
  \langle{x}\rangle^{\delta}f
  \rVert_{L^2(\mathbb{R}^{1+n})}.
\label{equation:chihara2}
\end{align}
\item[TYPE-II]
Suppose $1<m<n$. Then 
\begin{align}
  \lVert
  \langle{x}\rangle^{-m/2}
  \langle{D_x}\rangle^{(m-1)/2}
  e^{itp(D_x)}\phi
  \rVert_{L^2(\mathbb{R}^{1+n})}
& \leqslant
  C
  \lVert\phi\rVert_{L^2(\mathbb{R}^n)},
\label{equation:chihara3}
\\
  \lVert
  \langle{x}\rangle^{-m/2}
  \langle{D_x}\rangle^{m-1}
  Gf
  \rVert_{L^2(\mathbb{R}^{1+n})}
& \leqslant
  C
  \lVert
  \langle{x}\rangle^{m/2}f
  \rVert_{L^2(\mathbb{R}^{1+n})}.
\label{equation:chihara4}
\end{align}
\end{description} 
\end{theorem}
Theorem~\ref{theorem:main} follows from resolvent estimates. 
\begin{theorem}
\label{theorem:resolvent}
Let $n\geqslant2$. 
\begin{description}
\item[TYPE-I]
Suppose $m>1$ and $\delta>1/2$. Then 
\begin{equation}
\sup_{\zeta\in\mathbb{C}\setminus\mathbb{R}}
\Bigl\lvert
\Bigl(
\lvert{D_x}\rvert^{m-1}(\zeta-p(D_x))^{-1}f,g
\Bigr)_{L^2(\mathbb{R}^n)}
\Bigr\rvert
\leqslant
C
\lVert
\langle{x}\rangle^{\delta}f
\rVert_{L^2(\mathbb{R}^n)}
\lVert
\langle{x}\rangle^{\delta}g
\rVert_{L^2(\mathbb{R}^n)}.
\label{equation:resolvent1} 
\end{equation}
\item[TYPE-II]
Suppose $1<m<n$. Then 
\begin{equation}
\sup_{\zeta\in\mathbb{C}\setminus\mathbb{R}}
\Bigl\lvert
\Bigl(
\langle{D_x}\rangle^{m-1}(\zeta-p(D_x))^{-1}f,g
\Bigr)_{L^2(\mathbb{R}^n)}
\Bigr\rvert
\leqslant
C
\lVert
\langle{x}\rangle^{m/2}f
\rVert_{L^2(\mathbb{R}^n)}
\lVert
\langle{x}\rangle^{m/2}g
\rVert_{L^2(\mathbb{R}^n)}.
\label{equation:resolvent2} 
\end{equation}
\end{description}
\end{theorem}
To prove Theorem~\ref{theorem:resolvent}, 
we make full use of the estimates of the restriction of Fourier
transform on the level set of $a(\xi)$. 
Set $\Sigma(\tau)=\{\xi\in\mathbb{R}^n \vert a(\xi)=\tau\}$ for
$\tau\geqslant0$. 
The Fourier transform of $\phi(x)$ is denoted by 
$$
\hat{\phi}(\xi)
=
\frac{1}{(2\pi)^{n/2}}
\int_{\mathbb{R}^n}e^{-ix\cdot\xi}\phi(x)dx. 
$$
Kuroda first established the restriction estimates related with scattering
theory for $a(\xi)=\lvert\xi\rvert$ in \cite{kuroda}. 
His results played essential roles in \cite{ky}.  
His proof depends on the specificity of $a(\xi)=\lvert\xi\rvert$. 
We extend his results as follows. 
\begin{lemma}
\label{theorem:restriction}
Let $n\geqslant2$.  
\begin{description}
\item[Uniform trace estimates]
Suppose $\theta>0$. 
Then 
\begin{equation}
\lVert\hat{f}\rVert_{L^2(\Sigma(\tau))}
\leqslant
C
\lVert\langle{x}\rangle^{1/2+\theta}f\rVert_{L^2(\mathbb{R}^n)}.
\label{equation:trace} 
\end{equation}
\item[H\"older continuity]
Suppose $0<\theta\leqslant1/2$ for $n=2$, and 
$0<\theta<1$ for $n\geqslant3$. 
Then  
\begin{equation}
\lVert
\tau^{(n-1)/2}\hat{f}(\tau\cdot)
-
\lambda^{(n-1)/2}\hat{f}(\lambda\cdot)
\rVert_{L^2(\Sigma(1))}
\leqslant
C
\lvert\tau-\lambda\rvert^\theta
\lVert\langle{x}\rangle^{1/2+\theta}f\rVert_{L^2(\mathbb{R}^n)}.
\label{equation:hoelder}
\end{equation}
\item[Low frequency trace estimates]
Suppose $0<\theta<(n-1)/2$. Then 
\begin{equation}
\lVert\hat{f}\rVert_{L^2(\Sigma(\tau))}
\leqslant
C
\tau^\theta
\lVert\langle{x}\rangle^{1/2+\theta}f\rVert_{L^2(\mathbb{R}^n)}.
\label{equation:low-trace} 
\end{equation}
\end{description}
\end{lemma}
To conclude this section, we show how Theorem~\ref{theorem:resolvent} 
implies Theorem~\ref{theorem:main}. 
In \cite{kato} Kato discovered this principle 
in an abstract operator theoretic setting. 
We give an elementary Fourier analytic approach below.  
For an appropriate function $f(t,x)$, we use the following notation
\begin{align*}  \tilde{f}(\tau,\xi)
& =
  \frac{1}{(2\pi)^{(1+n)/2}}
  \iint_{\mathbb{R}^{1+n}}
  e^{-it\tau-ix\cdot\xi}
  f(t,x)
  dtdx,
\\
  \mathscr{F}_t[f](\tau,x)
& =
  \frac{1}{\sqrt{2\pi}}
  \int_{\mathbb{R}}
  e^{-it\tau}
  f(t,x)
  dt.
\end{align*}
Let $Y(t)$ be the Heaviside function
$$
Y(t)
=
\begin{cases}
1 & \ \text{if}\ \  t\geqslant0,
\\
0 & \ \text{if}\ \ t<0.
\end{cases}
$$
Set $f_\pm(t,x)=Y(\pm{t})f(t,x)$ for short. 
Using the Fourier transform in the space-time, we have 
$$
\widetilde{Gf}(\tau,\xi)
=
\frac{\tilde{f}_+(\tau,\xi)}{\tau-p(\xi)-i0}
+
\frac{\tilde{f}_-(\tau,\xi)}{\tau-p(\xi)+i0}.
$$
Using the above formula and the limiting absorption principle, 
one can easily show that 
\eqref{equation:resolvent1} implies \eqref{equation:chihara2},  
and 
\eqref{equation:resolvent2} implies \eqref{equation:chihara4}.  
On the other hand, 
the estimates of $e^{itp(D_x)}\phi$ are equivalent to 
$$
\lVert(Qe^{itp(D_x)})^\ast{f}\rVert_{L^2(\mathbb{R}^n)}
\leqslant
C
\lVert{f}\rVert_{L^2(\mathbb{R}^{1+n})} 
$$ 
since 
$$
\Bigl(Qe^{itp(D_x)}\phi,f\Bigr)_{L^2(\mathbb{R}^{1+n})}
=
\Bigl(\phi,(Qe^{itp(D_x)})^\ast{f}\Bigr)_{L^2(\mathbb{R}^n)}, 
$$
where 
$Q=\langle{x}\rangle^{-\delta}\lvert{D_x}\rvert^{(m-1)/2}$ 
or 
$Q=\langle{x}\rangle^{-m/2}\langle{D_x}\rangle^{(m-1)/2}$. 
By the co-area formula (see e.g., \cite[Theorem~5.8 in Chapter~II]{sakai}), 
H\"ormander's observation in \cite[Section~14.3]{hoermander} 
and the estimates \eqref{equation:resolvent1}-\eqref{equation:resolvent2}, 
we deduce 
\begin{align*}
& \quad
  \lVert(Qe^{itp(D_x)})^\ast{f}\rVert_{L^2(\mathbb{R}^n)}^2
\\
& =
  \lVert\widetilde{Q^\ast{f}}(p(\cdot),\cdot)\rVert_{L^2(\mathbb{R}^n)}^2
\\
& =
  \int_0^\infty
  \int_{p(\xi)=\tau}
  \frac{\lvert\widetilde{Q^\ast{f}}(\tau,\xi)\rvert^2}{\lvert{p^\prime(\xi)}\rvert}
  d\sigma(\xi)
  d\tau
\\
& =
  \lim_{\eta\downarrow0}
  \frac{1}{\pi}
  \int_0^\infty
  \int_\mathbb{R}
  \frac{\eta}{(\tau-\mu)^2+\eta^2}
  \int_{p(\xi)=\mu}
  \frac{\lvert\widetilde{Q^\ast{f}}(\tau,\xi)\rvert^2}{\lvert{p^\prime(\xi)}\rvert}
  d\sigma(\xi)
  d\mu
  d\tau
\\
& =
  \lim_{\eta\downarrow0}
  \im
  \frac{\pm}{\pi}
  \int_0^\infty
  \int_\mathbb{R}
  \frac{1}{\tau-\mu\pm{i}\eta}
  \int_{p(\xi)=\mu}
  \frac{\lvert\widetilde{Q^\ast{f}}(\tau,\xi)\rvert^2}{\lvert{p^\prime(\xi)}\rvert}
  d\sigma(\xi)
  d\mu
  d\tau  
\\
& =
  \lim_{\eta\downarrow0}
  \im
  \frac{\pm}{\pi}
  \int_0^\infty
  \int_{\mathbb{R}^n}
  \frac{1}{\tau\pm{i}\eta-p(\xi)}
  \frac{\lvert\widetilde{Q^\ast{f}}(\tau,\xi)\rvert^2}{\lvert{p^\prime(\xi)}\rvert}
  d\xi
  d\tau
\\
& =
  \lim_{\eta\downarrow0}
  \im
  \frac{\pm}{\pi}
  \int_0^\infty
  \Bigl(
  (\tau\pm{i}\eta-p(D_x))^{-1}Q^\ast\mathscr{F}_t[f](\tau),
  Q^\ast\mathscr{F}_t[f](\tau)
  \Bigr)_{L^2(\mathbb{R}^n)}
  d\tau
\\
& \leqslant
  C
  \int_0^\infty
  \lVert{\mathscr{F}_t[f](\tau)}\rVert_{L^2(\mathbb{R}^n)}^2
  d\tau
\\
& \leqslant
  C
  \lVert{f}\rVert_{L^2(\mathbb{R}^{1+n})}^2,
\end{align*}
which shows that 
\eqref{equation:resolvent1} implies \eqref{equation:chihara1}, 
and 
\eqref{equation:resolvent2} implies \eqref{equation:chihara3}. 
\par
The organization of this paper is as follows. 
In Section~2 we prepare some weighted commutator estimates needed
later. 
Section~3 is devoted to proving Lemma~\ref{theorem:restriction}. 
In Sections~4 and 5 we prove \eqref{equation:resolvent1} 
and \eqref{equation:resolvent2} respectively.  
\section{Weighted commutator estimates}
\label{section:commutator}
This section is devoted to preparing weighted commutator estimates used
later. Our basic tools are weighted estimates of fractional integrals
due to Stein and Weiss. 
\begin{theorem}[{\cite[Theorem~B$^\ast$]{sw}}]
Suppose 
$0<\alpha<n$, 
$\beta<n/2$, 
$\gamma<n/2$ 
and 
$\alpha=\beta+\gamma$. 
Then 
\begin{equation}
\lVert
a(\xi)^{-\beta}
\lvert{D_\xi}\rvert^{-\alpha}
\hat{f}
\rVert_{L^2(\mathbb{R}^n)}
\leqslant
C
\lVert
a(\xi)^\gamma
\hat{f}
\rVert_{L^2(\mathbb{R}^n)}. 
\label{equation:sw1}
\end{equation}
In particular, if $0\leqslant\beta<n/2$, then 
\begin{equation}
\lVert
a(\xi)^{-\beta}
\hat{f}
\rVert_{L^2(\mathbb{R}^n)}
\leqslant
C
\lVert
\lvert{x}\rvert^\beta
f
\rVert_{L^2(\mathbb{R}^n)}.  
\label{equation:sw2}
\end{equation}
\end{theorem}
In this section we show two lemmas. 
The first one is concerned with the commutator of weights and singular
integral operators of order zero. 
\begin{lemma}
\label{theorem:lem1}
Let $\delta$ satisfy 
$0<\delta\leqslant1$ for $n\geqslant3$ 
and 
$0<\delta<1$ for $n=2$. 
Suppose that 
$q(\xi){\in}C^\infty(\mathbb{R}^n\setminus\{0\})$ 
is homogeneous of degree zero. 
Then 
\begin{align}
  \lVert
  (\lvert{x}\rvert^\delta{q(D_x)}-q(D_x)\lvert{x}\rvert^\delta)f
  \rVert_{L^2(\mathbb{R}^n)}
& \leqslant
  C
  \lVert
  \lvert{x}\rvert^\delta
  f
  \rVert_{L^2(\mathbb{R}^n)},
\label{equation:com1}
\\
  \lVert
  \langle{x}\rangle^\delta
  q(D_x)
  f
  \rVert_{L^2(\mathbb{R}^n)}
& \leqslant
  C
  \lVert
  \langle{x}\rangle^\delta
  f
  \rVert_{L^2(\mathbb{R}^n)}.
\label{equation:com2}
\end{align}
\end{lemma}
\begin{proof}
It suffices to show \eqref{equation:com1} since 
\begin{align*}
  \lVert
  \langle{x}\rangle^\delta
  q(D_x)
  f
  \rVert_{L^2(\mathbb{R}^n)}  
& \leqslant
  C
  \lVert
  q(D_x)
  f
  \rVert_{L^2(\mathbb{R}^n)}
  +
  C
  \lVert
  \lvert{x}\rvert^\delta
  q(D_x)
  f
  \rVert_{L^2(\mathbb{R}^n)}
\\
& \leqslant
  C
  \lVert
  q(D_x)
  f
  \rVert_{L^2(\mathbb{R}^n)}
  +
  C
  \lVert
  q(D_x)
  \lvert{x}\rvert^\delta
  f
  \rVert_{L^2(\mathbb{R}^n)}
\\
& \qquad
  +
  C
  \lVert
  (
  \lvert{x}\rvert^\delta
  q(D_x)
  -
  q(D_x)
  \lvert{x}\rvert^\delta
  )
  f
  \rVert_{L^2(\mathbb{R}^n)}
\\
& \leqslant
  C
  \lVert
  \langle{x}\rangle^\delta
  f
  \rVert_{L^2(\mathbb{R}^n)}
  +
  \lVert
  (
  \lvert{x}\rvert^\delta
  q(D_x)
  -
  q(D_x)
  \lvert{x}\rvert^\delta
  )
  f
  \rVert_{L^2(\mathbb{R}^n)}.
\end{align*}
Since the inverse Fourier transform of $q(\xi)$ is 
a homogeneous function of degree $-n$, we have
\begin{equation}
\Bigl\lvert
(
\lvert{x}\rvert^\delta
q(D_x)
-
q(D_x)
\lvert{x}\rvert^\delta
)
f(x)
\Bigr\rvert
\leqslant
C
\int_{\mathbb{R}^n}
\frac{\Bigl\lvert{\lvert{x}\rvert^\delta-\lvert{y}\rvert^\delta}\Bigr\rvert}{\lvert{x-y}\rvert^n}
\lvert{f(y)}\rvert
dy.
\label{equation:e1}
\end{equation}
Pick up a positive integer $k$ satisfying $(k+1)\delta>1$
The mean value theorem gives 
\begin{align}
  \Bigl\lvert
  \lvert{x}\rvert^\delta
  -
  \lvert{y}\rvert^\delta
  \Bigr\rvert
& =
  \dfrac{
  \Bigl\lvert
  \lvert{x}\rvert^{(k+1)\delta}
  -
  \lvert{y}\rvert^{(k+1)\delta}
  \Bigr\rvert}
  {\displaystyle\sum_{j=0}^k
  \lvert{x}\rvert^{j\delta}\lvert{y}\rvert^{(k-j)\delta}}
\nonumber
\\
& =
  \frac{
  (k+1)\delta
  \Bigl\lvert
  \lvert{x}\rvert-\lvert{y}\rvert
  \Bigr\rvert}
  {\displaystyle\sum_{j=0}^k\lvert{x}\rvert^{j\delta}\lvert{y}\rvert^{(k-j)\delta}}
  \int_0^1
  \{t\lvert{x}\rvert+(1-t)\lvert{y}\rvert\}^{(k+1)\delta-1}
  dt
\nonumber
\\
& \leqslant
  \frac{(k+1)\delta\lvert{x-y}\rvert(\lvert{x}\rvert^{(k+1)\delta-1}+\lvert{y}\rvert^{(k+1)\delta-1})}{\lvert{x}\rvert^{k\delta}+\lvert{y}\rvert^{k\delta}}
\nonumber
\\ 
& \leqslant
  (k+1)\delta
  \frac{\lvert{x-y}\rvert}{\lvert{x}\rvert^{1-\delta}}
  +
  (k+1)\delta
  \frac{\lvert{x-y}\rvert}{\lvert{y}\rvert^{1-\delta}}
\label{equation:e2}
\end{align}
\par
We split our proof into two cases $n\geqslant3$ and $n=2$. 
When $n\geqslant3$, substituting \eqref{equation:e2} into
 \eqref{equation:e1}, we have 
\begin{align}
  \Bigl\lvert
  (
  \lvert{x}\rvert^\delta
  q(D_x)
  -
  q(D_x)
  \lvert{x}\rvert^\delta
  )
  f(x)
  \Bigr\rvert
& \leqslant
  C
  \int_{\mathbb{R}^n}
  \frac{1}{\lvert{x-y}\rvert^{n-1}}
  \left(
  \frac{1}{\lvert{x}\rvert^{1-\delta}}
  +
  \frac{1}{\lvert{y}\rvert^{1-\delta}}  
  \right)
  \lvert{f(y)}\rvert
  dy
\nonumber
\\
& =
  C
  \left\{
  \lvert{x}\rvert^{-(1-\delta)}
  \lvert{D_x}\rvert^{-1}
  \lvert{f(x)}\rvert
  +
  \lvert{D_x}\rvert^{-1}
  \Bigl(\lvert{x}\rvert^{-(1-\delta)}\lvert{f(x)}\rvert\Bigr)
  \right\}. 
\label{equation:e3}  
\end{align}
Applying \eqref{equation:sw2} with $\beta=1$ to \eqref{equation:e3}, 
we obtain \eqref{equation:com1} for $\delta=1$ and $n\geqslant3$. 
Suppose that $0<\delta<1$. 
Using 
\eqref{equation:sw1} with $(\alpha,\beta,\gamma)=(1,1-\delta,\delta)$ 
and 
\eqref{equation:sw2} with $\beta=1$, 
we get 
\begin{align}
  \Bigl\lVert
  \lvert{x}\rvert^{-(1-\delta)}
  \lvert{D_x}\rvert^{-1}
  \lvert{f(x)}\rvert
  \Bigr\rVert_{L^2(\mathbb{R}^n)}
& \leqslant
  C
  \Bigl\lVert
  \lvert{x}\rvert^\delta
  f
  \Bigr\rVert_{L^2(\mathbb{R}^n)},
\label{equation:e4}
\\
  \left\lVert
  \lvert{D_x}\rvert^{-1}
  \Bigl(\lvert{x}\rvert^{-(1-\delta)}\lvert{f(x)}\rvert\Bigr)
  \right\rVert_{L^2(\mathbb{R}^n)}
& =
  \left\lVert
  \lvert\xi\rvert^{-1}
  \lvert{D_\xi}\rvert^{-(1-\delta)}
  \widehat{\lvert{f}\rvert}
  \right\rVert_{L^2(\mathbb{R}^n)}
\nonumber
\\
& \leqslant
  C
  \left\lVert
  \lvert{x}\rvert^{1-(1-\delta)}
  \lvert{f}\rvert    
  \right\rVert_{L^2(\mathbb{R}^n)}
\nonumber
\\
& =
  C
  \lVert
  \lvert{x}\rvert^\delta
  f
  \rVert_{L^2(\mathbb{R}^n)}
\label{equation:e5}
\end{align}
respectively.  
Combining 
\eqref{equation:e3}, 
\eqref{equation:e4} 
and 
\eqref{equation:e5}, 
we obtain 
\eqref{equation:com1} 
for $0<\delta<1$ and $n\geqslant3$. 
\par
Suppose $n=2$ and $0<\delta<1$. 
Using elementary interpolation and \eqref{equation:e2}, we deduce  
\begin{align}
  \Bigl\lvert
  \lvert{x}\rvert^\delta-\lvert{y}\rvert^\delta
  \Bigr\rvert
& =
  \Bigl\lvert
  \lvert{x}\rvert^\delta-\lvert{y}\rvert^\delta
  \Bigr\rvert^{1-\delta}
  \Bigl\lvert
  \lvert{x}\rvert^\delta-\lvert{y}\rvert^\delta
  \Bigr\rvert^\delta
\nonumber
\\
& \leqslant
  \Bigl\lvert
  \lvert{x}\rvert^\delta+\lvert{y}\rvert^\delta
  \Bigr\rvert^{1-\delta}
  \Bigl\lvert
  \lvert{x}\rvert^\delta-\lvert{y}\rvert^\delta
  \Bigr\rvert^\delta
\nonumber
\\
& \leqslant
  C
  (\lvert{x}\rvert^\delta+\lvert{y}\rvert^\delta)^{1-\delta}
  \lvert{x-y}\rvert^\delta
  \left(
  \frac{1}{\lvert{x}\rvert^{1-\delta}}
  +
  \frac{1}{\lvert{y}\rvert^{1-\delta}}
  \right)^\delta
\nonumber
\\
& \leqslant
  C
  \lvert{x-y}\rvert^\delta
  \left(
  \lvert{x}\rvert^{\delta(1-\delta)}
  +
  \lvert{y}\rvert^{\delta(1-\delta)}
  \right)
  \left(
  \frac{1}{\lvert{x}\rvert^{\delta(1-\delta)}}
  +
  \frac{1}{\lvert{y}\rvert^{\delta(1-\delta)}}
  \right)
\nonumber
\\
& \leqslant
  C
  \lvert{x-y}\rvert^\delta
  \left(
  1
  +
  \frac{\lvert{y}\rvert^{\delta(1-\delta)}}{\lvert{x}\rvert^{\delta(1-\delta)}}
  +
  \frac{\lvert{x}\rvert^{\delta(1-\delta)}}{\lvert{y}\rvert^{\delta(1-\delta)}}
  \right)
\label{equation:e6}
\end{align}
Substituting \eqref{equation:e6} into \eqref{equation:e1}, we have 
\begin{align}
  \Bigl\lvert
  (
  \lvert{x}\rvert^\delta
  q(D_x)
  -
  q(D_x)
  \lvert{x}\rvert^\delta
  )
  f
  \Bigr\rvert
& \leqslant
  C
  \Bigl(
  \Bigl\lvert\lvert{D_x}\rvert^{-\delta}\lvert{f(x)}\rvert\Bigr\rvert
\nonumber
\\
& \qquad\qquad
  +
  \Bigl\lvert
  \lvert{x}\rvert^{\delta(1-\delta)}
  \lvert{D_x}\rvert^{-\delta}
  \lvert{x}\rvert^{-\delta(1-\delta)}
  \lvert{f(x)}\rvert
  \Bigr\rvert
\nonumber
\\
& \qquad\qquad\qquad
  +
  \Bigl\lvert
  \lvert{x}\rvert^{-\delta(1-\delta)}
  \lvert{D_x}\rvert^{-\delta}
  \lvert{x}\rvert^{\delta(1-\delta)}
  \lvert{f(x)}\rvert
  \Bigr\rvert 
  \Bigr).
\label{equation:e7}
\end{align}
Using \eqref{equation:sw2} with $\beta=\delta$, we have 
\begin{equation}
\Bigl\lVert
\lvert{D_x}\rvert^{-\delta}\lvert{f}\rvert
\Bigr\rVert_{L^2{\mathbb{R}^n}}
\leqslant
C
\Bigl\lVert
\lvert{x}\rvert^\delta
f
\Bigr\rVert_{L^2(\mathbb{R}^n)}.
\label{equation:e8}
\end{equation}
Here we remark that 
$0<\delta(1-\delta)\leqslant1/4$ 
and 
$\delta(2-\delta)<1=n/2$ 
for 
$0<\delta<1$. 
Using 
\eqref{equation:sw1} 
with 
$(\alpha,\beta,\gamma)=(\delta,-\delta(1-\delta),\delta(2-\delta))$ 
and 
$(\alpha,\beta,\gamma)=(\delta,\delta(1-\delta),\delta^2)$, 
we deduce 
\begin{align}
  \Bigl\lVert
  \lvert{x}\rvert^{(\delta(1-\delta))}
  \lvert{D_x}\rvert^{-\delta}
  \lvert{x}\rvert^{(-\delta(1-\delta))}
  \lvert{f}\rvert
  \Bigr\rVert_{L^2(\mathbb{R}^n)}
& \leqslant
  C
  \Bigl\lVert
  \lvert{x}\rvert^{\delta(2-\delta)-\delta(1-\delta)}
  \lvert{f}\rvert
  \Bigr\rVert_{L^2(\mathbb{R}^n)}
\nonumber
\\
& =
  C
  \Bigl\lVert
  \lvert{x}\rvert^\delta{f}
  \Bigr\rVert_{L^2(\mathbb{R}^n)},
\label{equation:e9}
\\
  \Bigl\lVert
  \lvert{x}\rvert^{-(\delta(1-\delta))}
  \lvert{D_x}\rvert^{-\delta}
  \lvert{x}\rvert^{(\delta(1-\delta))}
  \lvert{f}\rvert
  \Bigr\rVert_{L^2(\mathbb{R}^n)}
& \leqslant
  C
  \Bigl\lVert
  \lvert{x}\rvert^{\delta^2+\delta(1-\delta)}
  \lvert{f}\rvert
  \Bigr\rVert_{L^2(\mathbb{R}^n)}
\nonumber
\\
& =
  C
  \Bigl\lVert
  \lvert{x}\rvert^\delta{f}
  \Bigr\rVert_{L^2(\mathbb{R}^n)}.
\label{equation:e10}
\end{align}
Combining 
\eqref{equation:e7}, 
\eqref{equation:e8}, 
\eqref{equation:e9}, 
and 
\eqref{equation:e10}, 
we obtain  
\eqref{equation:com1} 
for 
$n=2$. 
This completes the proof. 
\end{proof}
The second lemma in this section is concerned with commutator estimates
between weights and fractional differentiations in frequency space of
$\xi\in\mathbb{R}^n$. 
\begin{lemma}
\label{theorem:lem2} 
Let $n\geqslant2$, and let $\kappa$ satisfy 
$0<\kappa<1$ for $n=2$ and $0<\kappa<3/2$ for $n\geqslant3$. 
Set $\rho=(n-1)/2$ and 
$r_\kappa(D_\xi)=\lvert{D_\xi}\rvert^{\kappa-1}D_\xi$. 
Then 
\begin{equation}
\left\lVert
a(\xi)^{-\rho}r_\kappa(D_\xi)a(\xi)^\rho\hat{f}
\right\rVert_{L^2(\mathbb{R}^n)} 
\leqslant
C
\Bigl\lVert
\lvert{x}\rvert^\kappa{f}
\Bigr\rVert_{L^2(\mathbb{R}^n)}. 
\label{equation:com3}
\end{equation}
\end{lemma}
\begin{proof}
First we remark that the Fourier transform of $r_\kappa(x)$ is 
homogeneous of degree $-(n+\kappa)$. 
We split our proof into four cases: 
Case~1 ($n\geqslant3$, $0<\kappa<1$), 
Case~2 ($n=2$, $0<\kappa<1$), 
Case~3 ($n\geqslant3$, $\kappa=1$) 
and  
Case~4 ($n\geqslant3$, $1<\kappa<3/2$). 
\\
{\bf Case~1}. 
Suppose $n\geqslant3$ and $0<\kappa<1$. 
Note that $\rho\geqslant1$. 
We compare $a(\xi)^{-\rho}r_\kappa(D_\xi)a(\xi)^\rho\hat{f}(\xi)$ 
with $r_\kappa(D_\xi)f(x)$. 
We evaluate 
\begin{equation}
\left\lvert
\Bigl\{
a(\xi)^{-\rho}r_\kappa(D_\xi)a(\xi)^\rho
-
r_\kappa(D_\xi)
\Bigr\}
\hat{f}(\xi)
\right\rvert
\leqslant
C
\int_{\mathbb{R}^n}
\left\lvert
\frac{a(\zeta)^\rho}{a(\xi)^\rho}-1
\right\rvert
\frac{\lvert\hat{f}(\zeta)\rvert}{\lvert\xi-\zeta\rvert^{n+\kappa}}
d\zeta.
\label{equation:e11}
\end{equation}
The mean value theorem gives 
\begin{align}
  \frac{a(\zeta)^\rho}{a(\xi)^\rho}-1
& =
  \frac{a(\zeta)^\rho-a(\xi)^\rho}{a(\xi)^\rho}
\nonumber
\\
& =
  (a(\zeta)-a(\xi))
  \frac{\rho}{a(\xi)^\rho}
  \int_0^1
  \Bigl\{ta(\zeta)+(1-t)a(\xi)\Bigr\}^{\rho-1}
  dt
\nonumber
\\
& =
  \frac{\rho}{a(\xi)^\rho}
  \int_0^1
  \Bigl\{ta(\zeta)+(1-t)a(\xi)\Bigr\}^{\rho-1}
  dt
\nonumber
\\
& \qquad
  \times
  \int_0^1
  a^\prime(s\zeta+(1-s)\xi)\cdot(\zeta-\xi)
  ds.
\label{equation:e12}
\end{align}
Since $\rho-1\geqslant0$, 
\begin{equation}
0
\leqslant
\max_{t\in[0,1]}
\Bigl\{ta(\zeta)+(1-t)a(\xi)\Bigr\}^{\rho-1}
\leqslant
C\{a(\zeta)^{\rho-1}+a(\xi)^{\rho-1}\}.
\label{equation:e13}
\end{equation}
Since $a^\prime(\xi)$ is homogeneous of degree zero, 
\begin{equation}
\lvert{a^\prime(s\zeta+(1-s)\xi)}\rvert
\leqslant
\max_{\omega\in\mathbb{S}^{n-1}}
\lvert{a^\prime(\omega)}\rvert
<+\infty, 
\label{equation:e14}
\end{equation}
where 
$
\mathbb{S}^{n-1}
=
\{\xi\in\mathbb{R}^n\ \vert \ \lvert\xi\rvert=1\}
$. 
Combining 
\eqref{equation:e12}, 
\eqref{equation:e13} 
and 
\eqref{equation:e14}, 
we have 
\begin{equation}
\left\lvert
\frac{a(\zeta)^\rho}{a(\xi)^\rho}-1
\right\rvert
\leqslant
C
\left(
\frac{a(\zeta)^{\rho-1}}{a(\xi)^\rho}
+
\frac{1}{a(\xi)}
\right)
\lvert\xi-\zeta\rvert. 
\label{equation:e15}
\end{equation}
Substituting \eqref{equation:e15} into \eqref{equation:e11}, we deduce 
\begin{align}
& \left\lvert
  \Bigl\{
  a(\xi)^{-\rho}{\lvert{D_\xi}\rvert^\kappa}a(\xi)^\rho
  -
  \lvert{D_\xi}\rvert^\kappa
  \Bigr\}
  \hat{f}(\xi)
  \right\rvert
\nonumber
\\
& \leqslant
  \frac{C}{a(\xi)^\rho}
  \int_{\mathbb{R}^n}
  \frac{a(\zeta)^{\rho-1}\lvert\hat{f}(\zeta)\rvert}{\lvert\xi-\zeta\rvert^{n-(1-\kappa)}}
  d\zeta
  +
  \frac{C}{a(\xi)}
  \int_{\mathbb{R}^n}
  \frac{\lvert\hat{f}(\zeta)\rvert}{\lvert\xi-\zeta\rvert^{n-(1-\kappa)}}
  d\zeta
\nonumber
\\
& =
  C
  \Bigl\lvert
  a(\xi)^{-\rho}\lvert{D_\xi}\rvert^{-(1-\kappa)}a(\xi)^{\rho-1}\lvert\hat{f}(\xi)\rvert
  \Bigr\rvert
  +
  C
  \Bigl\lvert
  a(\xi)^{-1}\lvert{D_\xi}\rvert^{-(1-\kappa)}\lvert\hat{f}(\xi)\rvert
  \Bigr\rvert
\label{equation:e16}
\end{align}
Using 
\eqref{equation:sw1} with 
$(\alpha,\beta,\gamma)=(1-\kappa,\rho,1-\kappa-\rho)$ 
and 
\eqref{equation:sw2} with $\beta=\kappa$, 
we deduce 
\begin{align}
  \Bigl\lVert
  a(\xi)^{-\rho}\lvert{D_\xi}\rvert^{-(1-\kappa)}a(\xi)^{\rho-1}
  \lvert\hat{f}\rvert
  \Bigr\rVert_{L^2(\mathbb{R}^n)} 
& \leqslant
  C
  \Bigl\lVert
  a(\xi)^{(1-\kappa-\rho)+(\rho-1)}
  \lvert\hat{f}\rvert
  \Bigr\rVert_{L^2(\mathbb{R}^n)}
\nonumber
\\
& =
  C
  \Bigl\lVert
  a(\xi)^{-\kappa}\hat{f}
  \Bigr\rVert_{L^2(\mathbb{R}^n)} 
\nonumber
\\
& \leqslant
  C
  \Bigl\lVert
  \lvert{x}\rvert^\kappa{f}
  \Bigr\rVert_{L^2(\mathbb{R}^n)}. 
\label{equation:e17}
\end{align}
Similarly, 
using \eqref{equation:sw1} with 
$(\alpha,\beta,\gamma)=(1-\kappa,1,-\kappa)$ 
and \eqref{equation:sw2} with $\beta=\kappa$, 
we deduce 
\begin{align}
  \Bigl\lVert
  a(\xi)^{-1}\lvert{D_\xi}\rvert^{-(1-\kappa)}
  \lvert\hat{f}\rvert
  \Bigr\rVert_{L^2(\mathbb{R}^n)} 
& \leqslant
  C
  \Bigl\lVert
  a(\xi)^{-\kappa}
  \lvert\hat{f}\rvert
  \Bigr\rVert_{L^2(\mathbb{R}^n)}
\nonumber
\\
& =
  C
  \Bigl\lVert
  a(\xi)^{-\kappa}\hat{f}
  \Bigr\rVert_{L^2(\mathbb{R}^n)} 
\nonumber
\\
& \leqslant
  C
  \Bigl\lVert
  \lvert{x}\rvert^\kappa{f}
  \Bigr\rVert_{L^2(\mathbb{R}^n)}. 
\label{equation:e18}
\end{align}
Applying 
\eqref{equation:e17} and \eqref{equation:e18} to \eqref{equation:e16}, 
we obtain \eqref{equation:com3} for $n\geqslant3$ and $0<\kappa<1$.
\\
{\bf Case~2}.
Suppose $n=2$ and $0<\kappa<1$. 
Note that $\rho=1/2$. 
We evaluate 
\begin{equation}
\left\lvert
\Bigl\{
a(\xi)^{-\rho}r_\kappa(D_\xi)a(\xi)^\rho
-
r_\kappa(D_\xi)
\Bigr\}
\hat{f}(\xi)
\right\rvert
\leqslant
C
\int_{\mathbb{R}^2}
\left\lvert
\frac{a(\zeta)^{1/2}}{a(\xi)^{1/2}}-1
\right\rvert
\frac{\lvert\hat{f}(\zeta)\rvert}{\lvert\xi-\zeta\rvert^{2+\kappa}}
d\zeta.
\label{equation:e19}
\end{equation}
Using factorization and the mean value theorem, we deduce 
\begin{align}
  \frac{a(\zeta)^{1/2}}{a(\xi)^{1/2}}-1 
& =
  \frac{a(\zeta)^{1/2}-a(\zeta)^{1/2}}{a(\xi)^{1/2}}
\nonumber
\\
& =
  \frac{a(\zeta)-a(\xi)}{a(\xi)^{1/2}\{a(\zeta)^{1/2}+a(\xi)^{1/2}\}}
\nonumber
\\
& =
  \frac{1}{a(\xi)^{1/2}\{a(\zeta)^{1/2}+a(\xi)^{1/2}\}}
\nonumber
\\
& \qquad
  \times
  \int_0^1
  a^\prime(s\zeta+(1-s)\xi)\cdot(\zeta-\xi)
  ds.
\end{align}
Then we have 
\begin{equation}
\left\lvert
\frac{a(\zeta)^{1/2}}{a(\xi)^{1/2}}-1
\right\rvert
\leqslant
\frac{C\lvert\xi-\zeta\rvert}{a(\xi)^{1/2}a(\zeta)^{1/2}}.
\label{equation:e20} 
\end{equation}
Applying \eqref{equation:e20} to \eqref{equation:e19}, we deduce 
\begin{equation}
\left\lvert
\Bigl\{
a(\xi)^{-1/2}r_\kappa(D_\xi)a(\xi)^{1/2}
-
r_\kappa(D_\xi)
\Bigr\}
\hat{f}(\xi)
\right\rvert
\leqslant
C
\left\lvert
a(\xi)^{-1/2}\lvert{D_\xi}\rvert^{1-\kappa}a(\xi)^{-1/2}
\lvert\hat{f}(\xi)\rvert
\right\rvert. 
\label{equation:e21}
\end{equation}
Applying 
\eqref{equation:sw1} with 
$(\alpha,\beta,\gamma)=(1-\kappa,1/2,1/2-\kappa)$ 
and 
\eqref{equation:sw2} with $\beta=\kappa$ to \eqref{equation:e21}, 
we deduce 
\begin{align*}
  \Bigl\lVert
  \Bigl\{
  a(\xi)^{-1/2}r_\kappa(D_\xi)a(\xi)^{1/2}
  -
  r_\kappa(D_\xi)
  \Bigr\}
  \hat{f}
  \Bigr\rVert_{L^2(\mathbb{R}^2)}
& \leqslant
  C
  \Bigl\lVert
  a(\xi)^{(1/2-\kappa)-1/2}\hat{f}
  \Bigr\rVert_{L^2(\mathbb{R}^2)}
\\
& =
  C
  \Bigl\lVert
  a(\xi)^{-\kappa}\hat{f}
  \Bigr\rVert_{L^2(\mathbb{R}^2)}
\\
& \leqslant
  C
  \Bigl\lVert
  \lvert{x}\rvert^\kappa{f}
  \Bigr\rVert_{L^2(\mathbb{R}^2)},
\end{align*}
which is desired.
\\
{\bf Case~3}. 
Suppose $n\geqslant3$ and $\kappa=1$. 
Since 
$$
\Bigl\{
a(\xi)^{-\rho}{D_\xi}a(\xi)^\rho-D_\xi
\Bigr\}
\hat{f}(\xi)
=
\rho
\frac{a^\prime(\xi)}{a(\xi)}
\hat{f}(\xi),
$$
using \eqref{equation:sw2} with $\beta=1$, we obtain 
\begin{align*}
  \Bigl\lVert
  \Bigl\{
  a(\xi)^{-\rho}{D_\xi}a(\xi)^\rho-D_\xi
  \Bigr\}
  \hat{f}
  \Bigr\rVert_{L^2(\mathbb{R}^n)}
& =
  \Bigl\lVert
  \rho{a^\prime(\xi)}a(\xi)^{-1}\hat{f}
  \Bigr\rVert_{L^2(\mathbb{R}^n)}
\\
& \leqslant
  \rho
  \max_{\omega\in\Sigma(1)}\lvert{a^\prime(\omega)}\rvert
  \Bigl\lVert
  a(\xi)^{-1}\hat{f}
  \Bigr\rVert_{L^2(\mathbb{R}^n)}
\\
& \leqslant
  C
  \Bigl\lVert
  \lvert{x}\rvert{f}
  \Bigr\rVert_{L^2(\mathbb{R}^n)}.
\end{align*}
\\
{\bf Case~4}. 
Suppose $n\geqslant3$ and $1<\kappa<3/2$. 
A simple computation gives 
\begin{align}
& \Bigl\{
  a(\xi)^{-\rho}\lvert{D_\xi}\rvert^{\kappa-1}D_\xi{a(\xi)^\rho}
  -
  \lvert{D_\xi}\rvert^{\kappa-1}D_\xi
  \Bigr\}
  \hat{f}(\xi)
\nonumber
\\
& =
  \Bigl\{
  a(\xi)^{-\rho}\lvert{D_\xi}\rvert^{\kappa-1}a(\xi)^\rho
  -
  \lvert{D_\xi}\rvert^{\kappa-1}
  \Bigr\}
  D_\xi\hat{f}(\xi)
\nonumber
\\
& +
  \rho{a(\xi)^{-\rho}}\lvert{D_\xi}\rvert^{\kappa-1}
  a(\xi)^{\rho-1}a^\prime(\xi)
  \hat{f}(\xi).
\label{equation:e22}
\end{align}
Using the results of Case~1, we deduce
\begin{align}
& \Bigl\lVert
  \Bigl\{
  a(\xi)^{-\rho}\lvert{D_\xi}\rvert^{\kappa-1}a(\xi)^\rho
  -
  \lvert{D_\xi}\rvert^{\kappa-1}
  \Bigr\}
  D_\xi\hat{f}
  \Bigr\rVert_{L^2(\mathbb{R}^n)}
\nonumber
\\
& \leqslant
  \Bigl\lVert
  \lvert{x}\rvert^{\kappa-1}xf
  \Bigr\rVert_{L^2(\mathbb{R}^n)}
  =
  \Bigl\lVert
  \lvert{x}\rvert^\kappa{f}
  \Bigr\rVert_{L^2(\mathbb{R}^n)},
  \label{equation:e23}
\end{align}
\begin{align}
& \Bigl\lVert
  \rho{a(\xi)^{-\rho}}\lvert{D_\xi}\rvert^{\kappa-1}
  a(\xi)^{\rho-1}a^\prime(\xi)
  \hat{f}
  \Bigr\rVert_{L^2(\mathbb{R}^n)} 
\nonumber
\\
& \leqslant
  \rho
  \Bigl\lVert
  \Bigl\{
  {a(\xi)^{-\rho}}\lvert{D_\xi}\rvert^{\kappa-1}a(\xi)^\rho
  -
  \lvert{D_\xi}\rvert^{\kappa-1}
  \Bigr\}
  \lvert\xi\rvert^{-1}b(\xi)
  \hat{f}
  \Bigr\rVert_{L^2(\mathbb{R}^n)} 
\nonumber
\\
& \qquad
  +
  \rho
  \Bigl\lVert
  \lvert\xi\rvert^{-1}b(\xi)\hat{f}
  \Bigr\rVert_{L^2(\mathbb{R}^n)} 
\nonumber
\\
& \leqslant
  C
  \Bigl\lVert
  \lvert{x}\rvert^{\kappa-1}b(D_x)\lvert{D_x}\rvert^{-1}f
  \Bigr\rVert_{L^2(\mathbb{R}^n)}, 
\label{equation:e24}
\end{align}
where $b(\xi)=a^\prime(\xi)\lvert\xi\rvert/a(\xi)$. 
Using \eqref{equation:com1} with $\delta=\kappa-1$, we have 
\begin{align}
& \Bigl\lVert
  \lvert{x}\rvert^{\kappa-1}b(D_x)\lvert{D_x}\rvert^{-1}f
  \Bigr\rVert_{L^2(\mathbb{R}^n)}
\nonumber
\\
& \leqslant
  \Bigl\lVert
  \Bigl\{
  \lvert{x}\rvert^{\kappa-1}b(D_x)
  -
  b(D_x)\lvert{x}\rvert^{\kappa-1}
  \Bigr\}
  \lvert{D_x}\rvert^{-1}f
  \Bigr\rVert_{L^2(\mathbb{R}^n)}
\nonumber
\\
& \qquad
  +
  \Bigl\lVert
   b(D_x)\lvert{x}\rvert^{\kappa-1}\lvert{D_x}\rvert^{-1}f
  \Bigr\rVert_{L^2(\mathbb{R}^n)}
\nonumber
\\
& \leqslant
  C
  \Bigl\lVert
  \lvert{x}\rvert^{\kappa-1}\lvert{D_x}\rvert^{-1}f
  \Bigr\rVert_{L^2(\mathbb{R}^n)}.
\label{equation:e25}  
\end{align}
Applying \eqref{equation:sw1} with 
$(\alpha,\beta,\gamma)=(1,-(\kappa-1),\kappa)$ 
to \eqref{equation:e25}, 
we have 
$$
\Bigl\lVert
\lvert{x}\rvert^{\kappa-1}\lvert{D_x}\rvert^{-1}f
\Bigr\rVert_{L^2(\mathbb{R}^n)}
\leqslant
\Bigl\lVert
\lvert{x}\rvert^\kappa{f}
\Bigr\rVert_{L^2(\mathbb{R}^n)}.
$$
Then we get 
\begin{equation}
\Bigl\lVert
\rho{a(\xi)^{-\rho}}
\lvert{D_\xi}\rvert^{\kappa-1}
a(\xi)^{\rho-1}a^\prime(\xi)
\hat{f}
\Bigr\rVert_{L^2(\mathbb{R}^n)}
\leqslant
C
\Bigl\lVert
\lvert{x}\rvert^\kappa{f}
\Bigr\rVert_{L^2(\mathbb{R}^n)}.
\label{equation:e26}
\end{equation}
Combining 
\eqref{equation:e22}, 
\eqref{equation:e23} 
and 
\eqref{equation:e26}, 
we obtain 
\eqref{equation:com3} 
for $n\geqslant3$ and $1<\kappa<3/2$. 
This completes the proof.
\end{proof}
\section{Restriction estimates}
\label{section:restriction}
In this section we prove Lemma~\ref{theorem:restriction}.
\begin{proof}[{\bf Proof of uniform trace estimates \eqref{equation:trace}}]
We split $\Sigma(\tau)$ into finite numbers of small surfaces 
given by graphs of functions as follows.  
Since 
$a^\prime(\xi){\in}C^\infty(\mathbb{R}^n\setminus\{0\})$, 
$a^\prime(\xi)\ne0$  for $\xi\ne0$, 
and 
$\Sigma(1)$ is compact, 
there exist finite numbers of closed sets 
$\Sigma(1,k)\subset\Sigma(1)$ 
and points 
$\omega_k\in\Sigma(1,k)$ ($k=1,\dotsc,l$) 
such that 
\begin{equation}
\Sigma(1)
=
\bigcup_{k=1}^l
\Sigma(1,k), 
\qquad
a^\prime(\omega_k){\cdot}a^\prime(\omega)
\geqslant
\frac{1}{2}\lvert{a^\prime(\omega_k)}\rvert^2
\quad\text{for}\quad \omega\in\Sigma(1,k). 
\label{equation:covering}
\end{equation}
Fix $k=1,\dotsc,l$. 
Using an appropriate rotation in $\xi\in\mathbb{R}^n$, 
we may assume 
$e_n=(0,\dotsc,0,1)=a^\prime(\omega_k)/\lvert{a^\prime(\omega_k)}\rvert$. 
Set 
$$
\Sigma(\tau,k)
=
\{
t\omega
\ \vert \ 
t>0, \ 
\omega\in\Sigma(1,k)
\}
\cap
\Sigma(\tau),
$$
$$
D(\tau,k)
=
\{
\xi^\prime\in\mathbb{R}^{n-1} 
\ \vert \ 
\text{there exists} \ 
\xi_n\in\mathbb{R}\ 
\text{such that} \ 
(\xi^\prime,\xi_n)\in\Sigma(\tau,k)
\}.
$$
The homogeneity of $a(\xi)$ implies that for $\tau>0$ 
$$
\Sigma(\tau,k)
=
\{
\tau\omega
\ \vert \ 
\omega\in\Sigma(1,k)
\},
\quad 
D(\tau,k)
=
\{\tau\xi^\prime \ \vert \ \xi^\prime{\in}D(1,k)\}
$$
Since $a(\xi)$ is positively homogeneous of degree one and 
\eqref{equation:covering}, 
$e_n$-direction is transversal to the small surface $\Sigma(\tau,k)$. 
Since 
$\lvert{\p{a(\omega_k)}/\p\xi_n}\rvert=\lvert{a^\prime(\omega_k)}\rvert\ne0$, 
the implicit function theorem shows that 
there exists a homogeneous function $g_k{\in}C^\infty(D(1,k))$ 
of degree zero such that 
\begin{equation}
\Sigma(1,k)
=
\{(\xi^\prime,g_k(\xi^\prime)) \ \vert \ \xi^\prime{\in}D(1,k)\} 
\end{equation}
provided that $\Sigma(1,k)$ is sufficiently small. 
Since $a(\xi^\prime,g_k(\xi^\prime))=1$ for $\xi^\prime{\in}D(1,k)$, 
we deduce 
$$
a\left(\xi^\prime,\tau{g_k}\left(\frac{\xi^\prime}{\tau}\right)\right)=\tau
\quad\text{for}\quad
\xi^\prime{\in}D(\tau,k).
$$
Set $g_{\tau,k}(\xi^\prime)=\tau{g_k}(\xi^\prime/\tau)$ for short. 
The uniqueness of the implicit function on $\Sigma(\tau,k)$ implies that 
$$
\Sigma(\tau,k)
=
\{(\xi^\prime,g_{\tau,k}(\xi^\prime))\ \vert \ \xi^\prime{\in}D(\tau,k)\}.
$$
\par
Now we evaluate $\lVert\hat{f}\rVert_{L^2(\Sigma(\tau,k))}$. 
Set $\nabla_{\xi^\prime}=\p/\p\xi^\prime$ for short. 
Note that 
$$
\nabla_{\xi^\prime}g_{\tau,k}(\xi^\prime)
=
(\nabla_{\xi^\prime}g_k)\left(\frac{\xi^\prime}{\tau}\right)
\quad\text{for}\quad
\xi^\prime{\in}D(\tau,k). 
$$
Then we have for any $\tau>0$
\begin{equation}
\max_{\xi^\prime{\in}D(\tau,k)}
\sqrt{1+\lvert\nabla_{\xi^\prime}g_{\tau,k}(\xi^\prime)\rvert^2}
=
\max_{\zeta^\prime{\in}D(1,k)}
\sqrt{1+\lvert\nabla_{\zeta^\prime}g_k(\zeta\prime)\rvert^2}
\equiv
M<+\infty.
\label{equation:jacobian} 
\end{equation}
Let $\sigma$ be the surface element on $\Sigma(\tau,k)$. 
Using \eqref{equation:jacobian}, 
the one-dimensional Sobolev embedding 
(See, e.g., \cite[Chapter~4]{taylor}) 
and 
the Plancherel formula, we deduce 
\begin{align*}
  \lVert\hat{f}\rVert_{L^2(\Sigma(\tau,k))}^2
& =
  \int_{\Sigma(\tau,k)}
  \lvert\hat{f}(\xi)\rvert^2
  d\sigma(\xi)
\\
& =
  \int_{D(\tau,k)}
  \lvert\hat{f}(\xi^\prime,g_{\tau,k}(\xi^\prime))\rvert^2
  \sqrt{1+\lvert\nabla_{\xi^\prime}g_{\tau,k}(\xi^\prime)\rvert^2}
  d\xi^\prime
\\
& \leqslant
  M 
  \int_{D(\tau,k)}
  \lvert\hat{f}(\xi^\prime,g_{\tau,k}(\xi^\prime))\rvert^2
  d\xi^\prime
\\
& \leqslant
  M 
  \int_{D(\tau,k)}
  \sup_{\xi_n\in\mathbb{R}}
  \lvert\hat{f}(\xi^\prime,\xi_n)\rvert^2
  d\xi^\prime
\\
& \leqslant
  C_{\theta,M}
  \int_{D(\tau,k)}
  \int_{\mathbb{R}}
  \lvert\langle{D_{\xi_n}}\rangle^{1/2+\theta}\hat{f}(\xi^\prime,\xi_n)\rvert^2
  d\xi^\prime
  d\xi_n
\\
& \leqslant
  C_{\theta,M}
  \int_{\mathbb{R}^n}
  \lvert\langle{D_{\xi_n}}\rangle^{1/2+\theta}\hat{f}(\xi)\rvert^2
  d\xi
\\
& =
  C_{\theta,M}
  \int_{\mathbb{R}^n}
  \lvert\langle{x_n}\rangle^{1/2+\theta}f(x)\rvert^2
  dx
\\
& \leqslant
  C_{\theta,M}
  \lVert\langle{x}\rangle^{1/2+\theta}f\rVert_{L^2(\mathbb{R}^n)}^2.
\end{align*}
Summing up the above estimates on $k$ up to $l$, we obtain \eqref{equation:trace}.  
\end{proof}
\begin{proof}[{\bf Proof of H\"older continuity \eqref{equation:hoelder}}]
In view of \eqref{equation:trace}, 
it suffices to show \eqref{equation:hoelder} only for 
$\lvert\tau-\lambda\rvert\leqslant1$. 
Fix $\omega\in\Sigma(1)$. 
The Sobolev embedding theorem shows that 
\begin{equation}
\Bigl\lvert
\tau^\rho\hat{f}(\tau\omega)-\lambda^\rho\hat{f}(\lambda\omega)
\Bigr\rvert 
\leqslant
C
\lvert\tau-\lambda\rvert^\theta
\left(
\int_0^\infty
\Bigl\{
\Bigl\lvert
\mu^\rho\hat{f}(\mu\omega)
\Bigr\rvert^2
+
\Bigl\lvert
\omega{\cdot}L_\mu
(\mu^\rho\hat{f}(\mu\omega))
\Bigr\rvert^2
\Bigr\}
d\mu
\right)^{1/2},
\label{equation:sobolev}
\end{equation}
where 
$L_\mu=r_{1/2+\theta}(D_\xi)$ for $\xi=\mu\omega$. 
Note that $d\xi=\mu^{2\rho}d\mu{d\sigma(\omega)}$ for $\xi=\mu\omega$. 
Integrating \eqref{equation:sobolev} over $\Sigma(1)$ 
and using \eqref{equation:com1}, we deduce 
\begin{align*}
& \Bigl\lVert
  \tau^\rho\hat{f}(\tau\cdot)-\lambda^\rho\hat{f}(\lambda\cdot)
  \Bigr\rVert_{L^2(\Sigma(1))}
\\
& \leqslant
  C
  \lvert\tau-\lambda\rvert^\theta 
  \left(
  \int_0^\infty
  \int_{\Sigma(1)}
  \Bigl\{
  \Bigl\lvert
  \mu^\rho\hat{f}(\mu\omega)
  \Bigr\rvert^2
  +
  \Bigl\lvert
  \omega{\cdot}L_u
  (\mu^\rho\hat{f}(\mu\omega))
  \Bigr\rvert^2
  \Bigr\}
  d\mu
  d\sigma(\omega)
  \right)^{1/2}
\\
& \leqslant
  C
  \lvert\tau-\lambda\rvert^\theta 
  \left(
  \int_0^\infty
  \int_{\Sigma(1)}
  \Bigl\{
  \Bigl\lvert
  \mu^\rho\hat{f}(\mu\omega)
  \Bigr\rvert^2
  +
  \Bigl\lvert
  L_u
  (\mu^\rho\hat{f}(\mu\omega))
  \Bigr\rvert^2
  \Bigr\}
  d\mu
  d\sigma(\omega)
  \right)^{1/2}
\\
& \leqslant
  C
  \lvert\tau-\lambda\rvert^\theta 
  \left\{
  \lVert\hat{f}\rVert_{L^2(\mathbb{R}^n)}
  +
  \Bigl\lVert
  a(\xi)^{-\rho}r_{1/2+\theta}(D_\xi)a(\xi)^\rho\hat{f}
  \Bigr\rVert_{L^2(\mathbb{R}^n)}
  \right\}
\\ 
& \leqslant
  C
  \lvert\tau-\lambda\rvert^\theta 
  \left\{
  \lVert{f}\rVert_{L^2(\mathbb{R}^n)}
  +
  \Bigl\lVert
  \lvert{x}\rvert^{1/2+\theta}f
  \Bigr\rVert_{L^2(\mathbb{R}^n)}
  \right\}
\\ 
& \leqslant
  C
  \lvert\tau-\lambda\rvert^\theta 
  \Bigl\lVert
  \langle{x}\rangle^{1/2+\theta}f
  \Bigr\rVert_{L^2(\mathbb{R}^n)}.
\end{align*}
This completes the proof.
\end{proof}
\begin{proof}[{\bf Proof of low frequency trace estimates \eqref{equation:low-trace}}]
In view of \eqref{equation:trace}, 
it suffices to show the case $\lvert\tau\rvert\leqslant1$. 
Suppose $0<\theta<1$ ($n=2$) or $0<\theta\leqslant1$ ($n\geqslant3$).  
Using \eqref{equation:hoelder}, we deduce
\begin{align}
  \lVert\hat{f}\rVert_{L^2(\Sigma(\tau))}
& =
  \Bigl\lVert\tau^\rho\hat{f}(\tau\cdot)\Bigr\rVert_{L^2(\Sigma(1))}
\nonumber
\\
& =
  \Bigl\lVert\tau^\rho\hat{f}(\tau\cdot)-0\Bigr\rVert_{L^2(\Sigma(1))}
\nonumber
\\
& \leqslant
  C
  \lvert\tau-0\rvert^\theta 
  \Bigl\lVert
  \langle{x}\rangle^{1/2+\theta}f
  \Bigr\rVert_{L^2(\mathbb{R}^n)}
\nonumber
\\
& \leqslant 
  C
  \lvert\tau\rvert^\theta 
  \Bigl\lVert
  \langle{x}\rangle^{1/2+\theta}f
  \Bigr\rVert_{L^2(\mathbb{R}^n)}.
\label{equation:e31}
\end{align}
\par
Consider the case $1\leqslant\theta<(n-1)/2$ ($n\geqslant4$). 
Set $\nu=\theta-1/2$ for short.  
Applying \eqref{equation:e31} with ``$\theta=1/2$'' to 
$\tau^{-\nu}\hat{f}(\tau\omega)$, 
we deduce 
\begin{align}
  \lVert\hat{f}\rvert_{L^2(\Sigma(\tau))}
& =
  \tau^\nu
  \Bigl\lVert\tau^{-\nu}\hat{f}(\tau\cdot)\Bigr\rVert_{L^2(\Sigma(\tau))}
\nonumber 
\\
& \leqslant
  C\tau^{\nu+1/2}
  \Bigl\lVert
  (1+\lvert{x}\rvert)
  a(D_x)^{-\nu}f 
  \Bigr\rVert_{L^2(\mathbb{R}^n)} 
\nonumber
\\
& \leqslant
  C\tau^{\theta}
  \left(
  \Bigl\lVert
  a(D_x)^{-\nu}f 
  \Bigr\rVert_{L^2(\mathbb{R}^n)} 
  +
  \Bigl\lVert
  \lvert{x}\rvert
  a(D_x)^{-\nu}f 
  \Bigr\rVert_{L^2(\mathbb{R}^n)} 
  \right).
\label{equation:e32}
\end{align}
Using \eqref{equation:sw2} with $\beta=\nu$, we have 
\begin{equation}
\Bigl\lVert
a(D_x)^{-\nu}f 
\Bigr\rVert_{L^2(\mathbb{R}^n)} 
\leqslant
C
\Bigl\lVert
\lvert{x}\rvert^{\theta-1/2}
\Bigr\rVert_{L^2(\mathbb{R}^n)}.
\label{equation:e33} 
\end{equation}
Using \eqref{equation:com1} 
with $\delta=1$ and $q(\xi)=a(\xi)^{-\nu}\lvert\xi\rvert^\nu$ 
and 
\eqref{equation:sw1} with $(\alpha,\beta,\gamma)=(\nu,-1,\nu+1)$, 
we deduce 
\begin{align}
  \Bigl\lVert
  \lvert{x}\rvert
  a(D_x)^{-\nu}f 
  \Bigr\rVert_{L^2(\mathbb{R}^n)} 
& =
  \Bigl\lVert
  \lvert{x}\rvert
  (a(D_x)^{-\nu}\lvert{D_x}\rvert^\nu)
  \lvert{D_x}\rvert^{-\nu}f 
  \Bigr\rVert_{L^2(\mathbb{R}^n)}
\nonumber
\\
& \leqslant
  C
  \Bigl\lVert
  \lvert{x}\rvert
  \lvert{D_x}\rvert^{-\nu}f 
  \Bigr\rVert_{L^2(\mathbb{R}^n)} 
\nonumber
\\
& \leqslant
  C
  \Bigl\lVert
  \lvert{x}\rvert^{\nu+1}f 
  \Bigr\rVert_{L^2(\mathbb{R}^n)} 
\nonumber
\\
& \leqslant
  C
  \Bigl\lVert
  \lvert{x}\rvert^{1/2+\theta}f 
  \Bigr\rVert_{L^2(\mathbb{R}^n)}.
\label{equation:e34}
\end{align}
Combining 
\eqref{equation:e32}, 
\eqref{equation:e33} 
and 
\eqref{equation:e34}, 
we obtain \eqref{equation:low-trace} for $1\leqslant\theta<(n-1)/2$. 
This completes the proof.
\end{proof}
%
%
\section{Smoothing estimates}
\label{section:smoothing}
In this section we prove the resolvent estimates \eqref{equation:resolvent1}. 
Obviously, it suffices to show \eqref{equation:resolvent1} for $1/2<\delta<1$. 
Let $b(\xi){\in}C^\infty(\mathbb{R}^n\setminus\{0\})$ be a real-valued symbol. 
We remark that for $\zeta\in\mathbb{C}\setminus\mathbb{R}$ 
\begin{align}
&  \Bigl(b(D_x)(\zeta-p(D_x))^{-1}f,g\Bigr)_{L^2(\mathbb{R}^n)}
\nonumber
\\
& =
  \frac{1}{4}
  \Bigl(b(D_x)(\zeta-p(D_x))^{-1}(f+g),(f+g)\Bigr)_{L^2(\mathbb{R}^n)}
\nonumber
\\
& +
  \frac{1}{4}
  \Bigl(b(D_x)(\zeta-p(D_x))^{-1}(f-g),(f-g)\Bigr)_{L^2(\mathbb{R}^n)}
\nonumber
\\
& +
  \frac{i}{4}
  \Bigl(b(D_x)(\zeta-p(D_x))^{-1}(f+ig),(f+ig)\Bigr)_{L^2(\mathbb{R}^n)}
\nonumber
\\
& +
  \frac{i}{4}
  \Bigl(b(D_x)(\zeta-p(D_x))^{-1}(f-ig),(f-ig)\Bigr)_{L^2(\mathbb{R}^n)}.
\label{equation:nanako}
\end{align}
Hence, it suffices to show \eqref{equation:resolvent1} for $f=g$. 
In view of \eqref{equation:com2} for 
$q(\xi)=\lvert{p^\prime(\xi)}\rvert^{-1/2}\lvert\xi\rvert^{(m-1)/2}$, 
the proof of \eqref{equation:resolvent1} is reduced to the following. 
\begin{lemma}
\label{theorem:lem23}
Let $n\geqslant2$. Suppose $m>1$ and $\delta>1/2$. Then, 
\begin{equation}
\sup_{\zeta\in\mathbb{C}\setminus\mathbb{R}}
\Bigl\lvert
\Bigl(
\lvert{p^\prime(D_x)}\rvert(\zeta-p(D_x))^{-1}f,g
\Bigr)_{L^2(\mathbb{R}^n)}
\Bigr\rvert
\leqslant
C
\lVert
\langle{x}\rangle^{\delta}f
\rVert_{L^2(\mathbb{R}^n)}
\lVert
\langle{x}\rangle^{\delta}g
\rVert_{L^2(\mathbb{R}^n)}.
\label{equation:resolvent4} 
\end{equation}
\end{lemma}
\begin{proof}
It suffices to show \eqref{equation:resolvent4} for $f=g$ and $1/2<\delta<1$. 
For the sake of convenience, 
set $\zeta=\lambda\pm{i}\eta$ with $\lambda\in\mathbb{R}$ and $\eta>0$. 
Using the Plancherel formula and the co-area formula, we have 

\begin{align}
& \Bigl(
  \lvert{p^\prime(D_x)}\rvert(\lambda\pm{i}\eta-p(D_x))^{-1}f,f
  \Bigr)_{L^2(\mathbb{R}^n)}
\nonumber
\\
& =
  \int_{\mathbb{R}^n}
  \frac{\lvert{p^\prime(\xi)}\rvert}{\lambda\pm{i}\eta-p(\xi)}
  \lvert\hat{f}(\xi)\rvert^2
  d\xi
\nonumber
\\
& =
  \int_0^\infty
  \frac{1}{(\lambda-\tau)\pm{i}\eta}
  \int_{p(\xi)=\tau}
  \lvert\hat{f}(\xi)\rvert^2
  d\sigma(\xi)
  d\tau
\nonumber
\\
& =
  \int_0^\infty
  \frac{(\lambda-\tau)\mp{i}\eta}{(\lambda-\tau)^2+\eta^2}
  \int_{p(\xi)=\tau}
  \lvert\hat{f}(\xi)\rvert^2
  d\sigma(\xi)
  d\tau.
\label{equation:e42}
\end{align}
Applying \eqref{equation:trace} to the imaginary part of \eqref{equation:e42}, 
we obtain 
\begin{align}
& \left\lvert
  \im
  \Bigl(
  \lvert{p^\prime(D_x)}\rvert(\lambda\pm{i}\eta-p(D_x))^{-1}f,f
  \Bigr)_{L^2(\mathbb{R}^n)}
  \right\rvert
\nonumber
\\
& =
  \int_0^\infty
  \frac{\eta}{(\lambda-\tau)^2+\eta^2}
  \int_{p(\xi)=\tau}
  \lvert\hat{f}(\xi)\rvert^2
  d\sigma(\xi)
  d\tau
\nonumber
\\
& =
  \int_0^\infty
  \frac{\eta}{(\lambda-\tau)^2+\eta^2}
  \lVert\hat{f}\rVert_{L^2(\Sigma(\tau^{1/m}))}^2
  d\tau
\nonumber
\\
& \leqslant
  C
  \lVert\langle{x}\rangle^\delta{f}\rVert_{L^2(\mathbb{R}^n)}^2
  \int_0^\infty
  \frac{\eta}{(\lambda-\tau)^2+\eta^2}
  d\tau
\nonumber
\\
& \leqslant
  C
  \lVert\langle{x}\rangle^\delta{f}\rVert_{L^2(\mathbb{R}^n)}^2
  \int_{\mathbb{R}}
  \frac{\eta}{(\lambda-\tau)^2+\eta^2}
  d\tau
\nonumber
\\
& =
  C\pi
  \lVert\langle{x}\rangle^\delta{f}\rVert_{L^2(\mathbb{R}^n)}^2.
\label{equation:e43}
\end{align}
According to $\lambda\in\mathbb{R}$, 
we split the evaluation of the real part of \eqref{equation:e42} 
into three cases: 
Case~1 ($\lambda\leqslant0$), 
Case~2 ($0<\lambda\leqslant2^m$) 
and 
Case~3 ($\lambda>2^m$). 
\\
{\bf Case~1}. 
Suppose $\lambda\leqslant0$. 
Set $\mu=-\lambda\geqslant0$. 
Using \eqref{equation:sw2} with $\beta=1/2$, we deduce 
\begin{align}
& \left\lvert
  \re
  \Bigl(
  \lvert{p^\prime(D_x)}\rvert(\lambda\pm{i}\eta-p(D_x))^{-1}f,f
  \Bigr)_{L^2(\mathbb{R}^n)}
  \right\rvert
\nonumber
\\
& =
  \int_0^\infty
  \frac{\mu+\tau}{(\mu+\tau)^2+\eta^2}
  \int_{p(\xi)=\tau}
  \lvert\hat{f}(\xi)\rvert^2
  d\sigma(\xi)
  d\tau
\nonumber
\\
& \leqslant
  \int_0^\infty
  \frac{1}{\mu+\tau}
  \int_{p(\xi)=\tau}
  \lvert\hat{f}(\xi)\rvert^2
  d\sigma(\xi)
  d\tau
\nonumber
\\
& \leqslant
  \int_0^\infty
  \frac{1}{\tau}
  \int_{p(\xi)=\tau}
  \lvert\hat{f}(\xi)\rvert^2
  d\sigma(\xi)
  d\tau
\nonumber
\\
& =
  \int_{\mathbb{R}^n}
  \frac{\lvert{p^\prime(\xi)}\rvert}{p(\xi)}
  \lvert\hat{f}(\xi)\rvert^2
  d\xi
\nonumber
\\
& \leqslant
  C
  \Bigl\lVert\lvert\xi\rvert^{-1/2}\hat{f}\Bigr\rVert^2_{L^2(\mathbb{R}^n)}
\nonumber
\\
& \leqslant
  C
  \Bigl\lVert\lvert{x}\rvert^{1/2}f\Bigr\rVert^2_{L^2(\mathbb{R}^n)}.
\label{equation:e44}
\end{align}
{\bf Case~2}. 
Suppose $0<\lambda\leqslant2^m$. 
We split the real part of \eqref{equation:e42} into three parts 
according to the size of $\tau$ as follows. 
\begin{align}
& \left\lvert
  \re
  \Bigl(
  \lvert{p^\prime(D_x)}\rvert(\lambda\pm{i}\eta-p(D_x))^{-1}f,f
  \Bigr)_{L^2(\mathbb{R}^n)}
  \right\rvert
\nonumber
\\
& =
  \left\lvert
  \int_0^\infty
  \frac{\lambda-\tau}{(\lambda-\tau)^2+\eta^2}
  \int_{p(\xi)=\tau}
  \lvert\hat{f}(\xi)\rvert^2
  d\sigma(\xi)
  d\tau
  \right\rvert
\nonumber
\\
& \leqslant
  \left\lvert
  \int_0^{\lambda/2}\dotsb{d}\tau 
  \right\rvert
  +
  \left\lvert
  \int_{\lambda/2}^{3\lambda/2}\dotsb{d}\tau 
  \right\rvert
  +
  \left\lvert
  \int_{3\lambda/2}^\infty\dotsb{d}\tau 
  \right\rvert 
\nonumber
\\
& =
  I_1+I_2+I_3.
\label{equation:e45}
\end{align}
It is easy to deal with $I_1$ and $I_3$. 
In fact, since 
$\lambda-\tau\geqslant\tau$ for $\tau\leqslant\lambda/2$, 
and 
$\tau-\lambda\geqslant\tau/3$ for $\tau\geqslant3\lambda/2$, 
we have 
$$
\frac{\lvert\lambda-\tau\rvert}{(\lambda-\tau)^2+\eta^2}
\leqslant
\frac{1}{\lvert\lambda-\tau\rvert}
\leqslant
\frac{3}{\tau}
$$
for $\tau\not\in(\lambda/2,3\lambda/2)$. 
In the same way as \eqref{equation:e44}, we can obtain 
\begin{equation}
I_1, I_3 
\leqslant 
C
\Bigl\lVert\lvert{x}\rvert^{1/2}f\Bigr\rVert^2_{L^2(\mathbb{R}^n)}. 
\label{equation:e46}
\end{equation}
\par
The estimate of $I_2$ is delicate. 
Since
\begin{equation}
\int_{\lambda/2}^{3\lambda/2}
\frac{\lambda-\tau}{(\lambda-\tau)^2+\eta^2}
d\tau
=
-
\int_{-\lambda/2}^{\lambda/2}
\frac{\mu}{\mu^2+\sigma^2}
d\mu
=0
\label{equation:vanishing}
\end{equation}
for $\lambda,\eta>0$, we have 
\begin{align}
  I_2
& =
  \left\lvert
  \int_{\lambda/2}^{3\lambda/2}
  \frac{\lambda-\tau}{(\lambda-\tau)^2+\eta^2}
  F_1(\tau,\lambda)
  d\tau
  \right\rvert,
\label{equation:e48}
\\
  F_1(\tau,\lambda)
& =
  \int_{p(\xi)=\tau}
  \lvert\hat{f}(\xi)\rvert^2
  d\sigma(\xi)
  -
  \int_{p(\xi)=\lambda}
  \lvert\hat{f}(\xi)\rvert^2
  d\sigma(\xi).
\nonumber
\end{align}
Applying \eqref{equation:trace} and 
\eqref{equation:hoelder} to $F_1(\tau,\lambda)$, 
we deduce 
\begin{align}
  \lvert{F_1(\tau,\lambda)}\rvert 
& =
  \Bigl\lvert
  \lVert\tau^{\rho/m}\hat{f}(\tau\cdot)\rVert_{L^2(\Sigma(1))}
  -
  \lVert\lambda^{\rho/m}\hat{f}(\lambda\cdot)\rVert_{L^2(\Sigma(1))}
  \Bigr\rvert
\nonumber
\\
& \qquad
  \times
  \Bigl(
  \lVert\hat{f}\rVert_{L^2(\Sigma(\tau^{1/m}))}
  +
  \lVert\hat{f}\rVert_{L^2(\Sigma(\lambda^{1/m}))}
  \Bigr)
\nonumber
\\
& \leqslant
  \lVert
  \tau^{\rho/m}\hat{f}(\tau\cdot)
  -
  \lambda^{\rho/m}\hat{f}(\lambda\cdot)
  \rVert_{L^2(\Sigma(1))}
\nonumber
\\
& \qquad
  \times
  \Bigl(
  \lVert\hat{f}\rVert_{L^2(\Sigma(\tau^{1/m}))}
  +
  \lVert\hat{f}\rVert_{L^2(\Sigma(\lambda^{1/m}))}
  \Bigr)
\nonumber
\\
& \leqslant
  C
  \lvert\tau^{1/m}-\lambda^{1/m}\rvert^\theta
  \lVert\langle{x}\rangle^\delta{f}\rVert_{L^2(\mathbb{R}^n)}^2,
\label{equation:e49}
\end{align}
where $\theta=\delta-1/2$. 
The mean value theorem shows that for 
$\lambda/2\leqslant\tau\leqslant3\lambda/2$ 
\begin{align}
  \lvert\tau^{1/m}-\lambda^{1/m}\rvert 
& =
  \frac{\lvert\tau-\lambda\rvert}{m}
  \int_0^1
  \Bigl(\lambda+t(\tau-\lambda)\Bigr)^{-(m-1)/m}
  dt
\nonumber
\\
& \leqslant
  \frac{\lvert\tau-\lambda\rvert}{m}
  \int_0^1
  \left(\frac{\lambda}{2}\right)^{-(m-1)/m}
  dt
\nonumber
\\
& =
  \frac{2^{(m-1)/m}}{m}\lambda^{-(m-1)/m}\lvert\tau-\lambda\rvert.
\label{equation:e50}
\end{align}
Substituting \eqref{equation:e50} into \eqref{equation:e49}, we have 
\begin{equation}
\lvert{F_1(\tau,\lambda)}\rvert
\leqslant
C
\lambda^{-\theta(m-1)/m}
\lvert\lambda-\tau\rvert^\theta
\lVert\langle{x}\rangle^\delta{f}\rVert_{L^2(\mathbb{R}^n)}^2
\label{equation:e47}
\end{equation}
for $\lambda/2\leqslant\tau\leqslant3\lambda/2$. 
Substituting \eqref{equation:e47} into \eqref{equation:e48}, we get 
\begin{align}
  I_2
& \leqslant
  C
  \lambda^{-\theta(m-1)/m}
  \lVert\langle{x}\rangle^\delta{f}\rVert_{L^2(\mathbb{R}^n)}^2
  \int_{\lambda/2}^{3\lambda/2}
  \lvert\lambda-\tau\rvert^{\theta-1}
  d\tau
\nonumber
\\
& =
  2C
  \lambda^{-\theta(m-1)/m}
  \lVert\langle{x}\rangle^\delta{f}\rVert_{L^2(\mathbb{R}^n)}^2
  \int_0^{\lambda/2}
  \mu^{\theta-1}
  d\mu
\nonumber
\\
& =
  \frac{2C}{\theta}
  \lambda^{\theta/m}
  \lVert\langle{x}\rangle^\delta{f}\rVert_{L^2(\mathbb{R}^n)}^2
\nonumber
\\
& \leqslant
  \frac{2^{1+\theta}C}{\theta}
  \lVert\langle{x}\rangle^\delta{f}\rVert_{L^2(\mathbb{R}^n)}^2.
\label{equation:e51}
\end{align}
\par
Combining \eqref{equation:e45}, \eqref{equation:e46} and \eqref{equation:e51}, 
we obtain \eqref{equation:resolvent4} for $f=g$ and $0<\re{\zeta}\leqslant2^m$. 
\\
{\bf Case~3}. 
Suppose $\lambda>2^m$. 
Case~3 is slightly different from Case~2. 
We split the real part of \eqref{equation:e42} into four parts as follows. 
\begin{align}
& \left\lvert
  \re
  \Bigl(
  \lvert{p^\prime(D_x)}\rvert(\lambda\pm{i}\eta-p(D_x))^{-1}f,f
  \Bigr)_{L^2(\mathbb{R}^n)}
  \right\rvert
\nonumber
\\
& =
  \left\lvert
  \int_0^\infty
  \frac{\lambda-\tau}{(\lambda-\tau)^2+\eta^2}
  \int_{p(\xi)=\tau}
  \lvert\hat{f}(\xi)\rvert^2
  d\sigma(\xi)
  d\tau
  \right\rvert
\nonumber
\\
& \leqslant
  \left\lvert
  \int_0^{\lambda-\lambda^{(m-1)/m}}\dotsb{d}\tau 
  \right\rvert
  +
  \left\lvert
  \int_{\lambda-\lambda^{(m-1)/m}}^{\lambda+\lambda^{(m-1)/m}}\dotsb{d}\tau 
  \right\rvert
\nonumber
\\
& \qquad 
  +
  \left\lvert
  \int_{\lambda+\lambda^{(m-1)/m}}^{2\lambda}\dotsb{d}\tau 
  \right\rvert 
  +
  \left\lvert
  \int_{2\lambda}^\infty\dotsb{d}\tau 
  \right\rvert 
\nonumber
\\
& =
  I_4+I_5+I_6+I_7.
\label{equation:e53}
\end{align}
It is easy to evaluate $I_4$, $I_6$ and $I_7$. 
Since 
$\lambda-\tau\geqslant\tau^{(m-1)/m}$ 
for 
$0\leqslant\tau\leqslant\lambda-\lambda^{(m-1)/m}$, 
and 
$\tau-\lambda\geqslant(\tau/2)^{(m-1)/m}$ 
for 
$\leqslant\lambda+\lambda^{(m-1)/m}\leqslant\tau\leqslant2\lambda$, 
$$
\frac{\lvert\lambda-\tau\rvert}{(\lambda-\tau)^2+\eta^2}
\leqslant
\frac{1}{\lvert\lambda-\tau\rvert}
\leqslant
\left(\frac{2}{\tau}\right)^{(m-1)/m}
$$
for 
$0\leqslant\tau\leqslant\lambda-\lambda^{(m-1)/m}$ 
and 
$\leqslant\lambda+\lambda^{(m-1)/m}\leqslant\tau\leqslant2\lambda$. 
Hence, we deduce 
\begin{align}
  I_4, I_6 
& \leqslant
  C
  \int_0^\infty
  \frac{1}{\tau^{(m-1)/m}}
  \int_{p(\xi)=\tau}
  \lvert\hat{f}(\xi)\rvert^2
  d\sigma(\xi)
  d\tau
\nonumber
\\
& =
  C
  \int_{\mathbb{R}^n}
  \frac{\lvert{p^\prime(\xi)}\rvert}{p(\xi)^{(m-1)/m}}
  \lvert\hat{f}(\xi)\rvert^2
  d\xi
\nonumber
\\
& \leqslant
  C
  \lVert{f}\rVert_{L^2(\mathbb{R}^n)}^2.
\label{equation:e54}
\end{align}
Since $\tau-\lambda\geqslant\tau/2$ for $\tau\geqslant2\lambda$, 
we can get 
\begin{equation}
I_7
\leqslant
\Bigl\lVert\lvert{x}\rvert^{1/2}f\Bigr\rVert^2_{L^2(\mathbb{R}^n)} 
\label{equation:e56}
\end{equation}
in the same way as \eqref{equation:e44}. 
\par
$I_5$ is also delicate.  
Since 
$$
\int_{\lambda-\lambda^{(m-1)/m}}^{\lambda+\lambda^{(m-1)/m}}
\frac{\lambda-\tau}{(\lambda-\tau)^2+\eta^2}
d\tau 
=
-
\int_{-\lambda^{(m-1)/m}}^{\lambda^{(m-1)/m}}
\frac{\mu}{\mu^2+\eta^2}
d\mu 
=
0, 
$$
we have 
$$
I_5
=
\left\lvert
\int_{\lambda-\lambda^{(m-1)/m}}^{\lambda+\lambda^{(m-1)/m}}
\frac{\lambda-\tau}{(\lambda-\tau)^2+\eta^2}
F_1(\tau,\lambda)
\tau 
\right\rvert.
$$
Here we remark that for $\lambda>2^m$ 
$$
(\lambda-\lambda^{(m-1)/m})-\frac{\lambda}{2}
=
\lambda^{(m-1)/m}
\left(\frac{\lambda^{1/m}}{2}-1\right)
>
0.
$$
Hence \eqref{equation:e50} is valid also for 
$\lambda>2^m$ and $\lvert\lambda-\tau\rvert\leqslant\lambda^{(m-1)/m}$, 
and 
\eqref{equation:e47} also holds for 
$\lambda>2^m$ and $\lvert\lambda-\tau\rvert\leqslant\lambda^{(m-1)/m}$. 
Thus, we can deduce 
\begin{equation}
I_5
\leqslant
C
\lambda^{-\theta(m-1)/m}
\lVert\langle{x}\rangle^\delta{f}\rVert_{L^2(\mathbb{R}^n)}^2
\int_0^{\lambda^{(m-1)/m}}
\mu^{\theta-1}
d\mu
=
\frac{C}{\theta}
\lVert\langle{x}\rangle^\delta{f}\rVert_{L^2(\mathbb{R}^n)}^2.
\label{equation:e59}
\end{equation}
\par
Combining \eqref{equation:e53}, \eqref{equation:e54}, 
\eqref{equation:e56} and \eqref{equation:e59}, 
we obtain \eqref{equation:resolvent4} for $f=g$ and $\re{\zeta}>2^m$. 
This completes the proof. 
\end{proof}
%
%
\section{Low frequency estimates}
\label{section:lowfrequency}
In this section we prove \eqref{equation:resolvent2}. 
In view of \eqref{equation:nanako}, 
it suffices to show \eqref{equation:resolvent2} only for $f=g$. 
We first show the estimates essentially 
related with the low frequency part. 
\begin{lemma}
\label{theorem:resolvent-low-frequency} 
Let $n\geqslant2$. 
Suppose $1<m<n$. Then 
\begin{equation}
\sup_{\zeta\in\mathbb{C}\setminus\mathbb{R}}
\Bigl\lvert
\Bigl(
(\zeta-p(D_x))^{-1}f,g
\Bigr)_{L^2(\mathbb{R}^n)}
\Bigr\rvert
\leqslant
C
\lVert
\langle{x}\rangle^{m/2}f
\rVert_{L^2(\mathbb{R}^n)}
\lVert
\langle{x}\rangle^{m/2}g
\rVert_{L^2(\mathbb{R}^n)}.
\label{equation:resolvent3} 
\end{equation}
\end{lemma}
\begin{proof}
In view of \eqref{equation:nanako}, 
we have only to show \eqref{equation:resolvent3} for $f=g$. 
Set $\zeta=\lambda\pm{i}\eta$ for $\lambda\in\mathbb{R}$ and $\eta>0$. 
In the same way as \eqref{equation:e42}, we have 
\begin{equation}
\Bigl(
(\zeta-p(D_x))^{-1}f,f
\Bigr)_{L^2(\mathbb{R}^n)}
=
  \int_0^\infty
  \frac{(\lambda-\tau)\mp{i}\eta}{(\lambda-\tau)^2+\eta^2}
  \int_{p(\xi)=\tau}
  \frac{\lvert\hat{f}(\xi)\rvert^2}{\lvert{p^\prime(\xi)}\rvert}
  d\sigma(\xi)
  d\tau. 
\label{equation:e60}
\end{equation}
\par
First we evaluate the imaginary part of \eqref{equation:e60}. 
Pick up $\theta\in(0,\min\{1/2,(m-1)/2\})$. 
Using \eqref{equation:low-trace}, we deduce 
\begin{align}
& \left\lvert
  \im
  \Bigl(
  (\lambda\pm{i}\eta-p(D_x))^{-1}f,f
  \Bigr)_{L^2(\mathbb{R}^n)}
  \right\rvert
\nonumber
\\
& =
  \int_0^\infty
  \frac{\eta}{(\lambda-\tau)^2+\eta^2}
  \int_{p(\xi)=\tau}
  \frac{\lvert\hat{f}(\xi)\rvert^2}{\lvert{p^\prime(\xi)}\rvert}
  d\sigma(\xi)
  d\tau
\nonumber
\\
& =
  \int_0^\infty
  \frac{\eta\tau^{-2\theta/m}}{(\lambda-\tau)^2+\eta^2}
  \int_{p(\xi)=\tau}
  \frac{\lvert\hat{f}(\xi)\rvert^2}{\lvert{p^\prime(\xi)}\rvert{a(\xi)^{-2\theta}}}
  d\sigma(\xi)
  d\tau
\nonumber
\\
& \leqslant
  C
  \int_0^\infty
  \frac{\eta\tau^{-2\theta/m}}{(\lambda-\tau)^2+\eta^2}
  \Bigl\lVert
  \lvert\xi\rvert^{-(m-1)/2+\theta}\hat{f}
  \Bigr\rVert^2_{L^2(\Sigma(\tau^{1/m}))}
  d\tau
\nonumber
\\
& \leqslant
  C
  \lVert
  \langle{x}\rangle^{1/2+\theta}
  \lvert{D_x}\rvert^{-(m-1)/2+\theta}f
  \rVert_{L^2(\mathbb{R}^n)}^2
  \int_0^\infty
  \frac{\eta}{(\lambda-\tau)^2+\eta^2}
  d\tau
\nonumber
\\
& \leqslant
  C
  \lVert
  \langle{x}\rangle^{1/2+\theta}
  \lvert{D_x}\rvert^{-(m-1)/2+\theta}f
  \rVert_{L^2(\mathbb{R}^n)}^2
  \int_{\mathbb{R}}
  \frac{\eta}{(\lambda-\tau)^2+\eta^2}
  d\tau
\nonumber
\\
& =
  C\pi
  \lVert
  \langle{x}\rangle^{1/2+\theta}
  \lvert{D_x}\rvert^{-(m-1)/2+\theta}f
  \rVert_{L^2(\mathbb{R}^n)}^2.
\label{equation:e61}
\end{align}
Here we remark that $0<(m-1)/2-\theta<m/2<n/2$. 
Using \eqref{equation:sw1} with 
$(\alpha,\beta,\gamma)=((m-1)/2-\theta,0,(m-1)/2-\theta)$ 
and 
$(\alpha,\beta,\gamma)=((m-1)/2-\theta,-1/2-\theta,m/2)$, 
we have 
\begin{align}
& \lVert
  \langle{x}\rangle^{1/2+\theta}
  \lvert{D_x}\rvert^{-(m-1)/2+\theta}f
  \rVert_{L^2(\mathbb{R}^n)} 
\nonumber
\\
& \leqslant
  C
  \Bigl\lVert
  \lvert{D_x}\rvert^{-(m-1)/2+\theta}f
  \Bigr\rVert_{L^2(\mathbb{R}^n)} 
  +
  C
  \Bigl\lVert
  \lvert{x}\rvert^{1/2+\theta}
  \lvert{D_x}\rvert^{-(m-1)/2+\theta}f
  \Bigr\rVert_{L^2(\mathbb{R}^n)} 
\nonumber
\\
& \leqslant
  C
  \Bigl\lVert
  \lvert{x}\rvert^{(m-1)/2-\theta}f
  \Bigr\rVert_{L^2(\mathbb{R}^n)} 
  +
  C
  \Bigl\lVert
  \lvert{x}\rvert^{m/2}f
  \Bigr\rVert_{L^2(\mathbb{R}^n)} 
\nonumber
\\
& \leqslant
  C
  \lVert
  \langle{x}\rangle^{m/2}f
  \rVert_{L^2(\mathbb{R}^n)}. 
\label{equation:e62}
\end{align}
Substituting \eqref{equation:e62} into \eqref{equation:e61}, we obtain 
\begin{equation}
\left\lvert
\im
\Bigl(
(\lambda\pm{i}\eta-p(D_x))^{-1}f,f
\Bigr)_{L^2(\mathbb{R}^n)}
\right\rvert
\leqslant
C
\lVert
\langle{x}\rangle^{m/2}f
\rVert_{L^2(\mathbb{R}^n)}^2. 
\label{equation:e63}
\end{equation}
\par
Next we consider the real part of \eqref{equation:e60}
\begin{align}
& \re
  \Bigl(
  (\lambda\pm{i}\eta-p(D_x))^{-1}f,f
  \Bigr)_{L^2(\mathbb{R}^n)}
\nonumber
\\
& =
  \int_0^\infty
  \frac{\lambda-\tau}{(\lambda-\tau)^2+\eta^2}
  \int_{p(\xi)=\tau}
  \frac{\lvert\hat{f}(\xi)\rvert^2}{\lvert{p^\prime(\xi)}\rvert}
  d\sigma(\xi)
  d\tau. 
\label{equation:e64}
\end{align}
When $\lambda\leqslant0$, set $\mu=-\lambda\geqslant0$. 
Using \eqref{equation:sw2} with $\beta=m/2$, we deduce 
\begin{align}
& \left\lvert
  \re
  \Bigl(
  (\lambda\pm{i}\eta-p(D_x))^{-1}f,f
  \Bigr)_{L^2(\mathbb{R}^n)}
  \right\rvert
\nonumber
\\
& =
  \int_0^\infty
  \frac{\mu+\tau}{(\mu+\tau)^2+\eta^2}
  \int_{p(\xi)=\tau}
  \frac{\lvert\hat{f}(\xi)\rvert^2}{\lvert{p^\prime(\xi)}\rvert}
  d\sigma(\xi)
  d\tau
\nonumber
\\
& \leqslant
  \int_0^\infty
  \frac{1}{\mu+\tau}
  \int_{p(\xi)=\tau}
  \frac{\lvert\hat{f}(\xi)\rvert^2}{\lvert{p^\prime(\xi)}\rvert}
  d\sigma(\xi)
  d\tau
\nonumber
\\
& \leqslant
  \int_0^\infty
  \frac{1}{\tau}
  \int_{p(\xi)=\tau}
  \frac{\lvert\hat{f}(\xi)\rvert^2}{\lvert{p^\prime(\xi)}\rvert}
  d\sigma(\xi)
  d\tau
\nonumber
\\
& =
  \int_{\mathbb{R}^n}
  p(\xi)^{-1}\lvert\hat{f}(\xi)\rvert^2
  d\xi
\nonumber
\\
& =
  \lVert
  a(\xi)^{-m/2}\hat{f}
  \rVert_{L^2(\mathbb{R}^n)}^2
\nonumber
\\
& \leqslant
  C
  \Bigl\lVert
  \lvert{x}\rvert^{m/2}f
  \Bigr\rVert_{L^2(\mathbb{R}^n)}^2
\label{equation:e65}
\end{align}
for $\lambda\leqslant0$ and $\eta>0$. 
\par
For $\lambda>0$, we split \eqref{equation:e64} into three parts 
\begin{align}
& \left\lvert
  \re
  \Bigl(
  (\lambda\pm{i}\eta-p(D_x))^{-1}f,f
  \Bigr)_{L^2(\mathbb{R}^n)}
  \right\rvert
\nonumber
\\
& =
  \left\lvert
  \int_0^\infty
  \frac{(\lambda-\tau)}{(\lambda-\tau)^2+\eta^2}
  \int_{p(\xi)=\tau}
  \frac{\lvert\hat{f}(\xi)\rvert^2}{\lvert{p^\prime(\xi)}\rvert}
  d\sigma(\xi)
  d\tau. 
  \right\rvert
\nonumber
\\
& \leqslant
  \left\lvert
  \int_0^{\lambda/2}\dotsb{d}\tau
  \right\rvert
  +
  \leqslant
  \left\lvert
  \int_{\lambda/2}^{3\lambda/2}\dotsb{d}\tau
  \right\rvert
  +
  \leqslant
  \left\lvert
  \int_{3\lambda/2}^\infty\dotsb{d}\tau
  \right\rvert
\nonumber
\\
& =
  I_8+I_9+I_{10}.
\label{equation:e66}
\end{align}
It is easy to handle $I_8$ and $I_{10}$. 
In fact, since 
$\lambda-\tau\geqslant\tau$ for $0\leqslant\tau\leqslant\lambda/2$ 
and 
$\tau-\lambda\geqslant\tau/3$ for $\tau\geqslant3\lambda/2$, 
$$
\frac{\lvert\lambda-\tau\rvert}{(\lambda-\tau)^2+\eta^2}
\leqslant
\frac{3}{\tau}
$$
for $\tau\not\in(\lambda/2,3\lambda/2)$. 
Hence we can obtain 
\begin{equation}
I_8, I_{10} 
\leqslant
C
\Bigl\lVert
\lvert{x}\rvert^{m/2}f
\Bigr\rVert_{L^2(\mathbb{R}^n)}^2
\end{equation}
for $\lambda>0$ in the same way as \eqref{equation:e65}. 
\par
We need to deal with $I_9$ carefully. 
Pick up $\theta\in(0,\min\{1/2,(m-1)/2\})$. 
Using \eqref{equation:vanishing}, we have 
\begin{equation}
I_9
=
\left\lvert
\int_{\lambda/2}^{3\lambda/2}
\frac{\lambda-\tau}{(\lambda-\tau)^2+\eta^2}
F_2(\tau,\lambda)
d\tau
\right\rvert,
\label{equation:e69} 
\end{equation}
\begin{align}
  F_2(\tau,\lambda)
& =
  \int_{p(\xi)=\tau}
  \frac{\lvert\hat{f}(\xi)\rvert^2}{\lvert{p^\prime(\xi)}\rvert}
  d\sigma(\xi)
  -
  \int_{p(\xi)=\lambda}
  \frac{\lvert\hat{f}(\xi)\rvert^2}{\lvert{p^\prime(\xi)}\rvert}
  d\sigma(\xi)
\nonumber
\\
& =
  \tau^{-2\theta/m}
  \int_{p(\xi)=\tau}
  \left\lvert
  \frac{\hat{f}(\xi)}{\lvert{p^\prime(\xi)}\rvert^{1/2}a(\xi)^{-\theta}}
  \right\rvert^2
  d\sigma(\xi)
\nonumber
\\  
& -
  \lambda^{-2\theta/m}
  \int_{p(\xi)=\lambda}
  \left\lvert
  \frac{\hat{f}(\xi)}{\lvert{p^\prime(\xi)}\rvert^{1/2}a(\xi)^{-\theta}}
  \right\rvert^2
  d\sigma(\xi)
\nonumber
\\
& =
  (\tau^{-2\theta/m}-\lambda^{-2\theta/m})
  \int_{p(\xi)=\tau}
  \left\lvert
  \frac{\hat{f}(\xi)}{\lvert{p^\prime(\xi)}\rvert^{1/2}a(\xi)^{-\theta}}
  \right\rvert^2
  d\sigma(\xi)
\nonumber
\\
& +
  \lambda^{-2\theta/m}
  \biggl(
  \int_{p(\xi)=\tau}
  \left\lvert
  \frac{\hat{f}(\xi)}{\lvert{p^\prime(\xi)}\rvert^{1/2}a(\xi)^{-\theta}}
  \right\rvert^2
  d\sigma(\xi)
\nonumber
\\
& \qquad \qquad
  -
  \int_{p(\xi)=\lambda}
  \left\lvert
  \frac{\hat{f}(\xi)}{\lvert{p^\prime(\xi)}\rvert^{1/2}a(\xi)^{-\theta}}
  \right\rvert^2
  d\sigma(\xi)
  \biggr)
\nonumber
\\
& =
  F_3(\tau,\lambda)+F_4(\tau,\lambda).
\label{equation:e70}
\end{align}
In the same way as \eqref{equation:e50}, we have 
\begin{equation}
\lvert\tau^{-2\theta/m}-\lambda^{-2\theta/m}\rvert 
\leqslant
C
\lambda^{-2\theta/m-1}
\lvert\tau-\lambda\rvert
\label{equation:e71}
\end{equation}
for $\lambda/2\leqslant\tau\leqslant3\lambda/2$. 
In the computation of \eqref{equation:e61} and \eqref{equation:e62}, 
we have obtained 
\begin{equation}
\int_{p(\xi)=\tau}
\left\lvert
\frac{\hat{f}(\xi)}{\lvert{p^\prime(\xi)}\rvert^{1/2}a(\xi)^{-\theta}}
\right\rvert^2
d\sigma(\xi)
\leqslant
C
\tau^{2\theta/m}
\lVert
\langle{x}\rangle^{m/2}f
\rVert_{L^2(\mathbb{R}^n)}^2. 
\label{equation:e72}
\end{equation}
Using \eqref{equation:e71} and \eqref{equation:e72}, we obtain 
\begin{equation}
\lvert{F_3(\tau,\lambda)}\rvert
\leqslant
C\lambda^{-1}\lvert\tau-\lambda\rvert
\label{equation:e73}
\end{equation}
for $\lambda/2\leqslant\tau\leqslant3\lambda/2$. 
Set 
$\hat{g}(\xi)=\hat{f}(\xi)/\lvert{p^\prime(\xi)}\rvert^{1/2}a(\xi)^{-\theta}$ 
for short. 
We remark that 
\eqref{equation:com2} and \eqref{equation:e62} show that for any $\tau>0$ 
\begin{equation}
\lVert
\langle{x}\rangle^{1/2+\theta}g
\rVert_{L^2(\mathbb{R}^n)}
\leqslant
C
\lVert
\langle{x}\rangle^{m/2}f
\rVert_{L^2(\mathbb{R}^n)}. 
\label{equation:help}
\end{equation}
Applying 
\eqref{equation:hoelder}, 
\eqref{equation:low-trace}, 
\eqref{equation:help} 
and 
\eqref{equation:e50}
to $F_4(\tau,\lambda)$, 
we deduce 
\begin{align}
  \lvert{F_4(\tau,\lambda)}\rvert
& =
  \lambda^{-2\theta/m}
  \Bigl\lvert
  \lVert\tau^{\rho/m}\hat{g}(\tau\cdot)\rVert_{L^2(\Sigma(1))}
  -
  \lVert\lambda^{\rho/m}\hat{g}(\lambda\cdot)\rVert_{L^2(\Sigma(1))}
  \Bigr\rvert
\nonumber
\\
& \qquad
  \times
  \Bigl(
  \lVert\hat{g}\rVert_{L^2(\Sigma(\tau^{1/m}))}
  +
  \lVert\hat{g}\rVert_{L^2(\Sigma(\lambda^{1/m}))}
  \Bigr)
\nonumber
\\
& \leqslant
  C
  \lambda^{-\theta/m}
  \lvert\tau^{1/m}-\lambda^{1/m}\rvert^\theta
  \lVert\langle{x}\rangle^{1/2+\theta}g\rVert_{L^2(\mathbb{R}^n)}^2
\nonumber
\\
& \leqslant
  C
  \lambda^{-\theta/m}
  \lvert\tau^{1/m}-\lambda^{1/m}\rvert^\theta
  \lVert\langle{x}\rangle^{m/2}f\rVert_{L^2(\mathbb{R}^n)}^2
\nonumber
\\
& \leqslant
  C
  \lambda^{-\theta}
  \lvert\tau-\lambda\rvert^\theta
  \lVert\langle{x}\rangle^{m/2}f\rVert_{L^2(\mathbb{R}^n)}^2
\label{equation:e74}
\end{align}
for $\lambda/2\leqslant\tau\leqslant3\lambda/2$. 
Combining \eqref{equation:e70}, \eqref{equation:e73} and \eqref{equation:e74}, 
we have 
\begin{equation}
\lvert{F_2(\tau,\lambda)}\rvert
\leqslant
C
(
\lambda^{-1}
\lvert\tau-\lambda\rvert
+
\lambda^{-\theta}
\lvert\tau-\lambda\rvert^\theta 
)
\lVert\langle{x}\rangle^{m/2}f\rVert_{L^2(\mathbb{R}^n)}^2
\label{equation:e78}
\end{equation}
for $\lambda/2\leqslant\tau\leqslant3\lambda/2$. 
Substituting \eqref{equation:e78} into \eqref{equation:e69}, we have 
\begin{align}
  I_9
& \leqslant
  C
  \lVert\langle{x}\rangle^{m/2}f\rVert_{L^2(\mathbb{R}^n)}^2
  \int_{\lambda/2}^{3\lambda/2}
  (\lambda^{-1}+\lambda^{-\theta}\lvert\tau-\lambda\rvert^{\theta-1})
  d\tau
\nonumber
\\
& =
  C
  \left(1+\frac{2^{1-\theta}}{\theta}\right)
  \lVert\langle{x}\rangle^{m/2}f\rVert_{L^2(\mathbb{R}^n)}^2
\label{equation:e79}
\end{align}
for $\lambda>0$. 
Combining 
\eqref{equation:e63}, 
\eqref{equation:e65} 
and  
\eqref{equation:e79}, 
we obtain \eqref{equation:resolvent3}. 
\end{proof}
Finally, we complete the proof of \eqref{equation:resolvent2}. 
\begin{proof}[Proof of \eqref{equation:resolvent2}] 
We make a little use of the elementary theory of 
pseudodifferential operators freely. 
See, e.g., \cite{kumano-go}. 
Pick up $\chi(\xi){\in}C^\infty(\mathbb{R}^n)$ satisfying 
$$
0\leqslant\chi(\xi)\leqslant1, 
\quad
\chi(\xi)
=
\begin{cases}
1 & (\lvert\xi\rvert\leqslant1),
\\
0 & (\lvert\xi\rvert\geqslant2).
\end{cases}
$$
We split $\langle\xi\rangle^{m-1}$ into two parts
\begin{equation}
\langle\xi\rangle^{m-1}
=
b_1(\xi)
+
\lvert\xi\rvert^{m-1}b_2(\xi), 
\label{equation:decomposition}
\end{equation}
$$
b_1(\xi)
=
\langle\xi\rangle^{m-1}\chi(\xi), 
\quad
b_2(\xi)
=
\frac{\langle\xi\rangle^{m-1}(1-\chi(\xi))}{\lvert\xi\rvert^{m-1}}.
$$
Here we remark that 
$b_1(\xi)$ and $b_2(\xi)$ are smooth functions on $\mathbb{R}^n$ 
whose derivatives of any order are all bounded. 
Using 
\eqref{equation:decomposition}, 
\eqref{equation:resolvent3} 
and 
\eqref{equation:resolvent1}, 
we deduce that 
\begin{align*}
& \Bigl\lvert
  \Bigl(
  \langle{D_x}\rangle^{m-1}(\zeta-p(D_x))^{-1}f,g
  \Bigr)_{L^2(\mathbb{R}^n)}
  \Bigr\rvert
\\
& \leqslant
  \Bigl\lvert
  \Bigl(
  (\zeta-p(D_x))^{-1}b_1(D_x)f,g
  \Bigr)_{L^2(\mathbb{R}^n)}
  \Bigr\rvert
\\
& +
  \Bigl\lvert
  \Bigl(
  \lvert{D_x}\rvert^{m-1}(\zeta-p(D_x))^{-1}b_2(D_x)f,g
  \Bigr)_{L^2(\mathbb{R}^n)}
  \Bigr\rvert
\\
& \leqslant
  C
  \lVert
  \langle{x}\rangle^{m/2}b_1(D_x)f
  \rVert_{L^2(\mathbb{R}^n)}
  \lVert
  \langle{x}\rangle^{m/2}g
  \rVert_{L^2(\mathbb{R}^n)}
\\
& +
  C
  \lVert
  \langle{x}\rangle^{m/2}b_2(D_x)f
  \rVert_{L^2(\mathbb{R}^n)}
  \lVert
  \langle{x}\rangle^{m/2}g
  \rVert_{L^2(\mathbb{R}^n)}. 
\end{align*}
Since 
$\langle{x}\rangle^{m/2}b_j(D_x)\langle{x}\rangle^{-m/2}$ ($j=1,2$) 
are $L^2$-bounded operators, 
we obtain \eqref{equation:resolvent2}. 
\end{proof}
%
%


\end{document}